\documentclass[reqno,10pt,centertags]{amsart} 
\usepackage{amsmath,amsthm,amscd,amssymb,latexsym,esint,upref,stmaryrd,
enumerate,color,verbatim,yfonts}
\usepackage{hyperref} 
\newcommand*{\mailto}[1]{\href{mailto:#1}{\nolinkurl{#1}}}
\newcommand{\arxiv}[1]{\href{http://arxiv.org/abs/#1}{arXiv:#1}}

\newcommand{\bbC}{{\mathbb{C}}}

\newcommand{\bbN}{{\mathbb{N}}}

\newcommand{\bbR}{{\mathbb{R}}}

\renewcommand{\a}{\alpha}
\renewcommand{\b}{\beta}
\newcommand{\g}{\gamma}
\renewcommand{\d}{\delta}

\DeclareMathOperator{\supp}{supp}

\DeclareMathOperator{\dom}{dom}

\renewcommand{\Re}{\text{\rm Re}}

\renewcommand{\ln}{\text{\rm ln}}

\newcommand{\norm}[1]{\lVert#1\rVert}

\newcommand{\no}{\notag}
\newcommand{\lb}{\label}
\newcommand{\f}{\frac}

\newcommand{\ol}{\overline}
\newcommand{\bs}{\backslash}

\newcommand{\wti}{\widetilde}
\newcommand{\Oh}{O}
\newcommand{\oh}{o}
\newcommand{\hatt}{\widehat} 
\newcommand{\dott}{\,\cdot\,}

\renewcommand{\dot}{\overset{\textbf{\Large.}}}
\renewcommand{\ddot}{\overset{\textbf{\Large..}}}

\newcommand{\bi}{\bibitem}

\newcommand{\al}{\alpha}
\newcommand{\be}{\beta}

\newcommand{\ve}{\varepsilon}

\newcommand{\Lr}{{L^2((a,b);rdx)}} 

\newcommand{\ACl}{{AC_{loc}((a,b))}}
\newcommand{\Ll}{{L^1_{loc}((a,b);dx)}}

\makeatletter
\def\theequation{\@arabic\c@equation}

\allowdisplaybreaks 
\numberwithin{equation}{section}

\newtheorem{theorem}{Theorem}[section]

\newtheorem{lemma}[theorem]{Lemma}
\newtheorem{corollary}[theorem]{Corollary}
\newtheorem{definition}[theorem]{Definition}
\newtheorem{hypothesis}[theorem]{Hypothesis}
\newtheorem{example}[theorem]{Example}

\theoremstyle{remark}
\newtheorem{remark}[theorem]{Remark}

\begin{document}

\title[Domain Properties of Bessel-Type Operators]{On Domain Properties of Bessel-Type Operators}

\author[F.\ Gesztesy]{Fritz Gesztesy}
\address{Department of Mathematics, 
Baylor University, Sid Richardson Bldg., 1410 S.~4th Street, Waco, TX 76706, USA}
\email{\mailto{Fritz\_Gesztesy@baylor.edu}}
\urladdr{\url{http://www.baylor.edu/math/index.php?id=935340}}

\author[M.\ M.\ H.\ Pang]{Michael M.\ H.\ Pang}
\address{Department of Mathematics,
University of Missouri, Columbia, MO 65211, USA}
\email{\mailto{pangm@missouri.edu}}
\urladdr{\url{https://www.math.missouri.edu/people/pang}}

\author[J.\ Stanfill]{Jonathan Stanfill}
\address{Department of Mathematics, 
Baylor University, Sid Richardson Bldg., 1410 S.~4th Street, Waco, TX 76706, USA}
\email{\mailto{Jonathan\_Stanfill@baylor.edu}}
\urladdr{\url{http://sites.baylor.edu/jonathan-stanfill/}}

\date{\today}
\dedicatory{Dedicated, with admiration, to Jerry Goldstein at the happy occasion of his 80th birthday.}
\@namedef{subjclassname@2020}{\textup{2020} Mathematics Subject Classification}

\begin{abstract}
Motivated by a recent study of Bessel operators in connection with a refinement of Hardy's inequality involving $1/\sin^2(x)$ on the finite interval $(0,\pi)$, we now take a closer look at the underlying Bessel-type operators with more general inverse square singularities at the interval endpoints. More precisely, we consider quadratic forms and operator realizations in $L^2((a,b); dx)$ associated with differential expressions of the form 
\[
\omega_{s_a} = - \f{d^2}{dx^2} + \f{s_a^2 - (1/4)}{(x-a)^2}, \quad 
s_a \in \bbR, \; x \in (a,b),   
\]
and 
\begin{align*}
\tau_{s_a,s_b} = - \f{d^2}{dx^2} + \f{s_a^2 - (1/4)}{(x-a)^2} + \f{s_b^2 - (1/4)}{(x-b)^2} + q(x), \quad x \in (a,b),& \\
s_a, s_b \in [0,\infty), \; q \in L^{\infty}((a,b); dx), \; q \text{ real-valued~a.e.~on $(a,b)$,}& 
\end{align*}
where $(a,b) \subset \bbR$ is a bounded interval. 

As an explicit illustration we describe the Krein--von Neumann extension of the minimal operator corresponding to 
$\omega_{s_a}$ and $\tau_{s_a,s_b}$. 
\end{abstract}

\maketitle

{\scriptsize{\tableofcontents}}
\normalsize

\section{Introduction} \lb{s1}

{\it As demonstrated, for instance, in the papers \cite{BG84}, \cite{BG84a}, \cite{GGR11}, \cite{GGR12}, \cite{GZ01}, Jerry has a longstanding interest in singular potentials of $|x|^{-2}$-type. 
Wishing Jerry a Very Happy Birthday, we hope our modest results on Bessel-type operators will cause him some joy.}

In a nutshell, the aim of this paper is to derive domain properties of Bessel-type operators associated with the singular second-order differential expressions of the form 
\begin{align}
\begin{split} 
\tau_{s_a,s_b} = - \f{d^2}{dx^2} + \f{s_a^2 - (1/4)}{(x-a)^2} + \f{s_b^2 - (1/4)}{(x-b)^2} + q(x), \quad x \in (a,b),& \\
s_a, s_b \in [0,\infty), \; q \in L^{\infty}((a,b); dx), \; q \text{ real-valued~a.e.~on $(a,b)$,}&
\end{split}     \lb{1.1} 
\end{align}
where $(a,b) \subset \bbR$ is a bounded interval. 

Our interest in $L^2((a,b); dx)$-realizations of $\tau_{s_a,s_b}$ has its origins in the linear operators (more precisely, the Friedrichs extensions of $\tau_{0,0}\big|_{C_0^{\infty}((a,b))}$) underlying the following Hardy-type inequalities 
\begin{align}
& \int_a^b dx \, |f'(x)|^2 \geq \frac{1}{4} \int_a^b dx \, \frac{|f(x)|^2}{x^2}, \quad f\in H_0^1 ((a,b)),   \lb{1.2} \\
& \int_a^b dx \, |f'(x)|^2 \geq \frac{1}{4} \int_a^b dx \, \frac{|f(x)|^2}{d_{(a,b)}(x)^2},\quad f\in H_0^1 ((a,b)), 
 \lb{1.3} \\
 \begin{split} 
&\int_a^b dx \, |f'(x)|^2 \geq \frac{\pi^2}{4(b-a)^2} \int_a^b dx \, \dfrac{|f(x)|^2}{\sin^2 (\pi(x-a)/(b-a))}    \\
& \hspace*{2.6cm} + \frac{\pi^2}{4(b-a)^2} \int_a^b dx \, |f(x)|^2,\quad f\in H_0^1 ((a,b)),     \lb{1.4} 
\end{split} 
\end{align}
recently studied in \cite{GPS21}.~Here $d_{(a,b)}(x)$ represents the distance from $x \in (a,b)$ to the 
boundary $\{a,b\}$ of the interval $(a,b)$, that is, 
\begin{align} \lb{1.5}
d_{(a,b)}(x)=\begin{cases} x-a, & x\in(a,(b+a)/2],\\
b-x, & x\in[(b+a)/2,b),
\end{cases}
\end{align}
and $H_0^1 ((a,b))$ is the standard Sobolev space on $(a,b)$ obtained upon completion of $C_0^\infty ((a,b))$ in the norm of $H^1 ((a,b))$, that is,
\begin{equation}
H_0^1 ((a,b)) = \big\{g \in L^2((a,b);dx)) \, \big| \, g \in AC([a,b]); \, g(0) =0; \, g'  
\in L^2((a,b);dx)) \big\}. 
\end{equation}

We emphasize that all constants in \eqref{1.2}--\eqref{1.4} are optimal and all inequalities are strict in the sense that equality holds in them if and only if $f \equiv 0$. We also stress that inequality \eqref{1.4} (for $a=0$, $b=\pi$) was first proved by Avkhadiev \cite[Lemma~1]{Av15}, a fact that was unfortunately missed in \cite{GPS21}.

Since $\tau_{s_a,s_b}$ in \eqref{1.1} permits the representation
\begin{align}
\begin{split} 
\tau_{s_a,s_b} &= \alpha_{s_a,s_b}^+  \alpha_{s_a,s_b}^{} + q(x) - \wti q(x), \quad x \in (a,b),      \lb{1.6} \\
& \hspace*{-8mm} 
s_a, s_b \in \bbR, \; q, \wti q \in L^{\infty}((a,b); dx), \; q, \wti q \text{ real-valued~a.e.~on $(a,b)$,}  
\end{split} 
\end{align}
where we abbreviated the singular first-order differential expressions 
\begin{align}
& \alpha_{s_a,s_b} = \f{d}{dx} - \f{s_a + (1/2)}{x-a} \wti \chi_{[a, a+\varepsilon]} 
+ \f{s_b + (1/2)}{b-x} \wti \chi_{[b-\varepsilon,b]},
\no \\
& \alpha_{s_a,s_b}^+ = - \f{d}{dx} - \f{s_a + (1/2)}{x-a} \wti \chi_{[a, a+\varepsilon]} 
+ \f{s_b + (1/2)}{b-x} \wti \chi_{[b-\varepsilon,b]},       \lb{1.7} \\
& \hspace*{5.32cm} s_a, s_b \in \bbR, \; x \in (a,b),     \no 
\end{align}
and for some fixed $\varepsilon \in (0, (b-a)/2)$, introduced the smooth step functions 
\begin{align}
\begin{split} 
& \wti \chi_{[a,a+\varepsilon]}, \wti \chi_{[b-\varepsilon,b]} \in C^{\infty}([a,b]),     \\
& \wti \chi_{[a,a+\varepsilon]}(x) = \begin{cases} 1, & x \in [a,a+\varepsilon], \\ 0, & x \in [a+2\varepsilon,b], \end{cases}  
\quad \wti \chi_{[b-\varepsilon,b]}(x) = \begin{cases} 1, & x \in [b-\varepsilon,b], \\ 0, & x \in [a,b-2\varepsilon], \end{cases}    \lb{1.8} 
\end{split} 
\end{align}
one is naturally led to the (pre)minimal and maximal operators in $L^2((a,b);dx)$ associated with $\alpha_{s_a,s_b}$, 
\begin{align}
& \dot A_{s_a,s_b,min} f = \alpha_{s_a,s_b} f, \quad s_a, s_b \in \bbR,      \no \\
& f \in\dom\big(\dot A_{s_a,s_b,min}\big) = \big\{g \in L^2((a,b); dx) \, \big| \, g \in AC_{loc}((a,b));   \\ 
&\hspace*{3.85cm} \supp(g) \subset (a,b) \text{ compact}; \,  \alpha_{s_a,s_b} g \in L^2((a,b); dx)\big\},   \no \\
& \ddot A_{s_a,s_b,min} = \alpha_{s_a,s_b}\big|_{C_0^{\infty}((a,b))}, \quad s_a, s_b \in \bbR,   \\
& A_{s_a,s_b,min} = \ol{\dot A_{s_a,s_b,min}} = \ol{\alpha_{s_a,s_b}\big|_{C_0^{\infty}((a,b))}}, 
\quad s_a, s_b \in \bbR,    \\
& A_{s_a,s_b,max} f = \alpha_{s_a,s_b} f, \quad s_a, s_b \in \bbR,   \no \\
& f \in\dom(A_{s_a,s_b,max}) = \big\{g \in L^2((a,b); dx) \, \big| \, g \in AC_{loc}((a,b));   \\ 
& \hspace*{5.8cm} \alpha_{s_a,s_b} g \in L^2((a,b); dx)\big\},   \no 
\end{align}
to be studied in Section \ref{s3}. 

Actually, due to locality, the inverse square singularities at $a$ and $b$ decouple and so can be studied separately. Temporarily focusing at the end point $a$, this amounts to studying the singular second-order differential expression
\begin{equation}
\omega_{s_a} = \alpha_{s_a}^+ \alpha_{s_a}^{} = - \f{d^2}{dx^2} + \f{s_a^2 - (1/4)}{(x-a)^2}, \quad 
s_a \in \bbR, \; x \in (a,b),   \lb{1.13} 
\end{equation}
where we abbreviated the singular first-order differential expressions 
\begin{align}
\begin{split} 
& \alpha_{s_a} = \f{d}{dx} - \f{s_a + (1/2)}{x-a} = (x-a)^{s_a + (1/2)} \f{d}{dx} (x-a)^{- s_a - (1/2)}, \\
& \alpha_{s_a}^+ = - \f{d}{dx} - \f{s_a + (1/2)}{x-a} = - \alpha_{-s_a - 1}, \quad 
s_a \in \bbR, \; x \in (a,b).      \lb{1.14}
\end{split} 
\end{align}
Again, this naturally leads one to the (pre)minimal and maximal operators associated with $\alpha_{s_a}$, 
$s_a \in \bbR$,  
\begin{align}
& \dot A_{s_a,min} f = \alpha_{s_a} f, \quad s_a \in \bbR,      \no \\
& f \in\dom\big(\dot A_{s_a,min}\big) = \big\{g \in L^2((a,b); dx) \, \big| \, g \in AC_{loc}((a,b)); \, 
 \alpha_{s_a} g \in L^2((a,b); dx);  \no \\
&\hspace*{7.2cm} \supp(g) \subset (a,b) \text{ compact} \big\},    \lb{1.15}   \\
& \ddot A_{s_a,min} = \alpha_{s_a}\big|_{C_0^{\infty}((a,b))}, \quad s_a \in \bbR,    \\ 
& A_{s_a,min} = \ol{\dot A_{s_a,min}} = \ol{\alpha_{s_a}\big|_{C_0^{\infty}((a,b))}}, \quad s_a \in \bbR,    \\
& A_{s_a,max} f = \alpha_{s_a} f, \quad s_a \in \bbR,         \lb{1.18} \\
& f \in\dom(A_{s_a,max}) = \big\{g \in L^2((a,b); dx) \, \big| \, g \in AC_{loc}((a,b)); \, 
 \alpha_{s_a} g \in L^2((a,b); dx)\big\},   \no 
\end{align}
also to be studied in Section \ref{s3}. 

Interchanging the role of the endpoints $a$ and $b$ then leads to the analogous singular second-order differential expression  
\begin{equation}
\eta_{s_b} = \beta_{s_b}^+ \beta_{s_b}^{} = - \f{d^2}{dx^2} + \f{s_b^2 - (1/4)}{(b-x)^2}, \quad s_b \in \bbR, \; x \in (a,b),   \lb{1.19} 
\end{equation}
abbreviating the singular first-order differential expressions
\begin{align}
\begin{split} 
& \beta_{s_b} = \f{d}{dx} + \f{s_b + (1/2)}{b-x} = (b-x)^{s_b + (1/2)} \f{d}{dx} (b-x)^{- s_b - (1/2)}, \\
& \beta_{s_b}^+ = - \f{d}{dx} + \f{s_b + (1/2)}{b-x} = - \beta_{-s_b-1}, \quad 
s_b \in \bbR, \; x \in (a,b).      \lb{1.20}
\end{split} 
\end{align}
The associated (pre)minimal and maximal operators in $L^2((a,b);dx)$ associated with $\beta_{s_b}$ are then of the form
\begin{align}
& \dot B_{s_b,min} f = \beta_{s_b} f, \quad s_b \in \bbR,      \no \\
& f \in\dom\big(\dot B_{s_b,min}\big) = \big\{g \in L^2((a,b); dx) \, \big| \, g \in AC_{loc}((a,b)); \, 
 \beta_{s_b} g \in L^2((a,b); dx);  \no \\
&\hspace*{6.8cm} \supp(g) \subset (a,b) \text{ compact} \big\},       \\
& \ddot B_{s_b,min} = \beta_{s_b}\big|_{C_0^{\infty}((a,b))}, \quad s_b \in \bbR,   \\
& B_{s_b,min} = \ol{\dot B_{s_b,min}} = \ol{\beta_{s_b}\big|_{C_0^{\infty}((a,b))}}, \quad s_b \in \bbR,    \\
& B_{s_b,max} f = \beta_{s_b} f, \quad s_b \in \bbR,      \\
& f \in\dom(B_{s_b,max}) = \big\{g \in L^2((a,b); dx) \, \big| \, g \in AC_{loc}((a,b)); \, 
 \beta_{s_b} g \in L^2((a,b); dx)\big\},   \no 
\end{align}
and their various properties can then be read off the ones for $\dot A_{s_a,min}$, $\ddot A_{s_a,min}$, 
$A_{s_a,min}$, and $A_{s_a,max}$.

While these singular first-order differential operators are studied in depth in Section \ref{s3}, the Bessel-type differential operators associated with the $L^2((a,b);dx)$-realizations of the second-order differential expressions 
\begin{equation}
\tau_{s_a,s_b} = \alpha_{s_a,s_b}^+  \alpha_{s_a,s_b}^{} + q(\dott) - \wti q(\dott), \quad 
\omega_{s_a} = \alpha_{s_a}^+ \alpha_{s_a}^{}, \quad 
\eta_{s_b} = \beta_{s_b}^+ \beta_{s_b}^{},  
\end{equation} 
are the subject of Section \ref{s4}, where we discuss the minimal and maximal operators $A_{s_a,s_b,min}$, $A_{s_a,s_b,max}$, $A_{s_a,min}$, $A_{s_a,max}$ and then focus on the Friedrichs extensions $A_{s_a,s_b,F}$, $A_{s_a,F}$ and the associated form domains thereof. In Section \ref{s5} we focus on the associated Krein--von Neumann extensions.

Section \ref{s2} recalls the necessary background for the singular Weyl--Titchmarsh--Kodaira theory as needed in the bulk of this paper and Appendix \ref{sA} considers the particular case of the Krein--von Neumann extension in the special case $q=0$ in \eqref{1.6}. Finally, Appendix \ref{sB} summarizes variants of Hardy's inequality as employed in Section \ref{s3}. 

Due to the enormity of the literature on Bessel-type operators, an exhaustive survey of the literature on this subject is an insurmountable task. Hence, we conclude this introduction with a brief discussion of the literature in the immediate vicinity of the circle of ideas discussed in this paper (for additional comments in this direction we refer the reader to \cite{GPS21}). For some background on Bessel operators and their spectral properties we refer, for instance, to 
\cite[p.~544--552]{AG81}, \cite{DW11}, \cite[p.~1532--1538]{DS88}, \cite{EK07}, \cite{GNS21}, 
\cite[p.~142--144]{Na68}, \cite[p.~81--90]{Ti62}, and the literature cited therein. Closest to the results presented in this paper is the work by A.\ Yu.\ Annan'eva and V.\ S.\ Budyka \cite{AA12}, \cite{AB16}, \cite{AB15}, and by L.\ Bruneau, J.\ Derezinski, and V.\ Georgescu \cite{BDG11}, \cite{DG21}, Derezinski and Faupin \cite{DF21}, which focus on Bessel operators on $(0, \infty)$ and $(0,b)$, $b \in (0,\infty)$, with inverse square singularity at $x=0$. The techniques used by these authors are based on elements of harmonic analysis and the explicit knowledge of resolvents for first-order differential operators, respectively, while our methods predominantly rely on weighted Hardy inequalities.

\section{Some Background}\lb{s2}

In this section, following \cite{GLN20}, \cite{GLPS21}, and \cite[Ch.~13]{GZ21}, we summarize the singular 
Weyl--Titchmarsh--Kodaira theory as needed to treat the Bessel operators in the remainder of this paper. 

Throughout this section we make the following assumptions:

\begin{hypothesis} \lb{h2.1}
Let $(a,b) \subseteq \bbR$ and suppose that $p,q,r$ are $($Lebesgue\,$)$ measurable functions on $(a,b)$ 
such that the following items $(i)$--$(iii)$ hold: \\[1mm] 
$(i)$ \hspace*{1.1mm} $r>0$ a.e.~on $(a,b)$, $r\in\Ll$. \\[1mm] 
$(ii)$ \hspace*{.1mm} $p>0$ a.e.~on $(a,b)$, $1/p \in\Ll$. \\[1mm] 
$(iii)$ $q$ is real-valued a.e.~on $(a,b)$, $q\in\Ll$. 
\end{hypothesis}

Given Hypothesis \ref{h2.1}, we study Sturm--Liouville operators associated with the general, 
three-coefficient differential expression
\begin{equation}
\tau=\f{1}{r(x)}\left[-\f{d}{dx}p(x)\f{d}{dx} + q(x)\right] \, \text{ for a.e.~$x\in(a,b) \subseteq \bbR$,} 
   \lb{2.1}
\end{equation} 
and introduce maximal and minimal operators in $L^2((a,b);rdx)$ associated with $\tau$ in the usual manner as follows: 
\begin{align}
&T_{max} f = \tau f,    \no
\\
& f \in \dom(T_{max})=\big\{g\in L^2((a,b);rdx) \, \big| \,g,g^{[1]} \in\ACl;   \lb{2.2} \\ 
& \hspace*{6.35cm}  \tau g\in L^2((a,b);rdx)\big\},   \no
\end{align}
where 
\begin{equation}
y^{[1]}(x) = p(x) y'(x), \quad x \in (a,b),
\end{equation}
denotes the first quasi-derivative of a function $y\in AC_{loc}((a,b))$. 
The {\it preminimal operator} $\dot T_{min} $ in $L^2((a,b);rdx)$ associated with $\tau$ is defined by 
\begin{align}
& \dot T_{min}  f = \tau f,   \no
\\
&f \in \dom \big(\dot T_{min}\big)=\big\{g\in L^2((a,b);rdx) \, \big| \, g,g^{[1]} \in\ACl;   \lb{2.3}
\\
&\hspace*{3.15cm} \supp \, (g)\subset(a,b) \text{ is compact; } \tau g\in L^2((a,b);rdx)\big\}.   \no
\end{align}

One can prove that $\dot T_{min} $ is closable, and one then defines the {\it minimal operator} $T_{min}$ as the 
closure of $\dot T_{min} $. The following fact is well known, 
\begin{equation} 
\big(\dot T_{min}\big)^* = T_{max}, 
\end{equation} 
and hence $T_{max}$ is closed. Moreover, $\dot T_{min} $ is essentially self-adjoint if and only if $T_{max}$ is symmetric, and then $\ol{\dot T_{min} }=T_{min}=T_{max}$.

The celebrated Weyl alternative then can be stated as follows:

\begin{theorem}[Weyl's Alternative] \lb{t2.2} ${}$ \\
Assume Hypothesis \ref{h2.1}. Then the following alternative holds: Either \\[1mm] 
$(i)$ for every $z\in\bbC$, all solutions $u$ of $(\tau-z)u=0$ are in $\Lr$ near $b$ 
$($resp., near $a$$)$, \\[1mm] 
or, \\[1mm] 
$(ii)$ for every $z\in\bbC$, there exists at least one solution $u$ of $(\tau-z)u=0$ which is not in $\Lr$ near $b$ $($resp., near $a$$)$. In this case, for each $z\in\bbC\bs\bbR$, there exists precisely one solution $u_b$ $($resp., $u_a$$)$ of $(\tau-z)u=0$ $($up to constant multiples$)$ which lies in $\Lr$ near $b$ $($resp., near $a$$)$. 
\end{theorem}

This naturally yields the limit circle/limit point classification of $\tau$ at an interval endpoint as follows. 

\begin{definition} \lb{d2.3} 
Assume Hypothesis \ref{h2.1}. \\[1mm]  
In case $(i)$ in Theorem \ref{t2.2}, $\tau$ is said to be in the \textit{limit circle case} at $b$ $($resp., at $a$$)$. $($Frequently, $\tau$ is then called \textit{quasi-regular} at $b$ $($resp., $a$$)$.$)$
\\[1mm] 
In case $(ii)$ in Theorem \ref{t2.2}, $\tau$ is said to be in the \textit{limit point case} at $b$ $($resp., at $a$$)$. \\[1mm]
If $\tau$ is in the limit circle case at $a$ and $b$ then $\tau$ is also called \textit{quasi-regular} on $(a,b)$. 
\end{definition}

The next result links self-adjointness of $T_{min}$ (resp., $T_{max}$) and the limit point property of $\tau$ at both endpoints:

\begin{theorem} \lb{t2.4}
Assume Hypothesis~\ref{h2.1}, then the following items $(i)$ and $(ii)$ hold: \\[1mm] 
$(i)$ If $\tau$ is in the limit point case at $a$ $($resp., $b$$)$, then 
\begin{equation} 
W(f,g)(a)=0 \, \text{$($resp., $W(f,g)(b)=0$$)$ for all $f,g\in\dom(T_{max})$.} 
\end{equation} 
$(ii)$ Let $T_{min}=\ol{\dot T_{min} }$. Then
\begin{align}
\begin{split}
n_\pm(T_{min}) &= \dim(\ker(T_{max} \mp i I))    \\
& = \begin{cases}
2 & \text{if $\tau$ is in the limit circle case at $a$ and $b$,}\\
1 & \text{if $\tau$ is in the limit circle case at $a$} \\
& \text{and in the limit point case at $b$, or vice versa,}\\
0 & \text{if $\tau$ is in the limit point case at $a$ and $b$}.
\end{cases}
\end{split} 
\end{align}
In particular, $T_{min} = T_{max}$ is self-adjoint if and only if $\tau$ is in the limit point case at $a$ and $b$. 
\end{theorem}

Here the Wronskian of $f$ and $g$, for $f,g\in\ACl$, is defined by
\begin{equation}
W(f,g)(x) = f(x)g^{[1]}(x) - f^{[1]}(x)g(x), \quad x \in (a,b).    \lb{23.2.3.1} \no \\
\end{equation}

We now make the additional assumption of boundedness from below of $\dot T_{min}$, that is, strengthen Hypothesis \ref{h2.1} as follows: 

\begin{hypothesis} \lb{h2.5} 
In addition to Hypothesis \ref{h2.1} suppose that also the following item $(iv)$ holds: \\[1mm]
$(iv)$ There exists $\lambda_0 \in \bbR$ such that $\dot T_{min} \geq \lambda_0 I_{L^2((a,b);rdx)}$ 
$($equivalently, $T_{min} \geq \lambda_0 I_{L^2((a,b);rdx)}$$)$. 
\end{hypothesis}

Assuming Hypothesis \ref{h2.5} holds from now on, for fixed $c, d \in (a,b)$, $c \leq d$, $\tau u = \lambda u$, 
$\lambda \leq \lambda_0$ has real-valued nonvanishing solutions $u_a(\lambda,\dott) \neq 0$,
$\hatt u_a(\lambda,\dott) \neq 0$ in the neighborhood $(a,c]$ of $a$, and real-valued nonvanishing solutions
$u_b(\lambda,\dott) \neq 0$, $\hatt u_b(\lambda,\dott) \neq 0$ in the neighborhood $[d,b)$ of
$b$, such that 
\begin{align}
&W(\hatt u_a (\lambda,\dott),u_a (\lambda,\dott)) = 1,
\quad u_a (\lambda,x)=\oh(\hatt u_a (\lambda,x))
\text{ as $x\downarrow a$,} \lb{2.8} \\
&W(\hatt u_b (\lambda,\dott),u_b (\lambda,\dott))\, = 1,
\quad u_b (\lambda,x)\,=\oh(\hatt u_b (\lambda,x))
\text{ as $x\uparrow b$,} \lb{2.9} \\
&\int_a^c dx \, p(x)^{-1}u_a(\lambda,x)^{-2}=\int_d^b dx \, 
p(x)^{-1}u_b(\lambda,x)^{-2}=\infty,  \lb{2.10} \\
&\int_a^c dx \, p(x)^{-1}{\hatt u_a(\lambda,x)}^{-2}<\infty, \quad 
\int_d^b dx \, p(x)^{-1}{\hatt u_b(\lambda,x)}^{-2}<\infty. \lb{2.11}
\end{align}
In this case, $u_a(\lambda,\dott)$ (resp., $u_b(\lambda,\dott)$) is called a {\it principal} solution of 
$\tau u=\lambda u$ at $a$ (resp., $b$) and any real-valued solution $\wti{\wti u}_a(\lambda,\dott)$ 
(resp., $\wti{\wti u}_b(\lambda,\dott)$) of $\tau u=\lambda u$ linearly independent of $u_a(\lambda,\dott)$ (resp.,
$u_b(\lambda,\dott)$) is called {\it nonprincipal} at $a$ (resp., $b$). In particular, $\hatt u_a (\lambda,\dott)$ 
(resp., $\hatt u_b(\lambda,\dott)$) in \eqref{2.8}--\eqref{2.11} are nonprincipal solutions at $a$ (resp., $b$).  

One then shows the existence of the following boundary values for all elements $g \in \dom(T_{max})$, 
\begin{align}
\begin{split} 
\wti g(a) &=- W(u_a(\lambda_0,\dott), g)(a)= \lim_{x \downarrow a} \f{g(x)}{\hatt u_a(\lambda_0,x)},    \\
\wti g(b) &=- W(u_b(\lambda_0,\dott), g)(b)= \lim_{x \uparrow b} \f{g(x)}{\hatt u_b(\lambda_0,x)}, 
\lb{2.12} 
\end{split} \\
\begin{split} 
{\wti g}^{\, \prime}(a) &= W(\hatt u_a(\lambda_0,\dott), g)(a)= \lim_{x \downarrow a} \f{g(x) - \wti g(a) \hatt u_a(\lambda_0,x)}{u_a(\lambda_0,x)},    \\ 
{\wti g}^{\, \prime}(b) &= W(\hatt u_b(\lambda_0,\dott), g)(b)= \lim_{x \uparrow b} \f{g(x) - \wti g(b) \hatt u_b(\lambda_0,x)}{u_b(\lambda_0,x)},    \lb{2.13}
\end{split} 
\end{align}
and obtains the following characterization for all self-adjoint extensions of $T_{min}$ (resp., all self-adjoint restrictions of $T_{max}$).

\begin{theorem} \lb{t2.6}
Assume Hypothesis \ref{h2.5} and that $\tau$ is in the limit circle case at $a$ and $b$ $($i.e., $\tau$ is quasi-regular 
on $(a,b)$$)$. Then the following items $(i)$--$(iii)$ hold: \\[1mm] 
$(i)$ All self-adjoint extensions $T_{\al,\be}$ of $T_{min}$ with separated boundary conditions are of the form
\begin{align}
& T_{\al,\be} f = \tau f, \quad \al,\be\in[0,\pi),   \no \\
& f \in \dom(T_{\al,\be})=\big\{g\in\dom(T_{max}) \, \big| \, \wti g(a)\cos(\al) + {\wti g}^{\, \prime}(a)\sin(\al)=0;   \lb{2.14} \\ 
& \hspace*{5.5cm} \, \wti g(b)\cos(\be) + {\wti g}^{\, \prime}(b)\sin(\be) = 0 \big\}.    \no 
\end{align}
$(ii)$ All self-adjoint extensions $T_{\varphi,R}$ of $T_{min}$ with coupled boundary conditions are of the type
\begin{align}
\begin{split} 
& T_{\varphi,R} f = \tau f,    \\
& f \in \dom(T_{\varphi,R})=\bigg\{g\in\dom(T_{max}) \, \bigg| \begin{pmatrix} \wti g(b)\\ {\wti g}^{\, \prime}(b)\end{pmatrix} 
= e^{i\varphi}R \begin{pmatrix}
\wti g(a)\\ {\wti g}^{\, \prime}(a) \end{pmatrix} \bigg\}, \lb{2.15}
\end{split}
\end{align}
where $\varphi\in[0,2\pi)$, and $R$ is a real $2\times2$ matrix with $\det(R)=1$ 
$($i.e., $R \in SL(2,\bbR)$$)$.  \\[1mm] 
$(iii)$ Every self-adjoint extension of $T_{min}$ is either of type $(i)$ $($i.e., separated\,$)$ or of type 
$(ii)$ $($i.e., coupled\,$)$.
\end{theorem}

\begin{remark} \lb{r2.7}
If $\tau$ is in the limit point case at one endpoint, say, at the endpoint $b$, one omits the corresponding boundary condition involving $\beta \in [0, \pi)$ at $b$ in \eqref{2.14} to obtain all self-adjoint extensions $T_{\alpha}$ of 
$T_{min}$, indexed by $\alpha \in [0, \pi)$. (In this case item $(ii)$ in Theorem \ref{t2.6} is vacuous.) In the case where $\tau$ is in the limit point case at both endpoints, all boundary values and boundary conditions become superfluous as in this case $T_{min} = T_{max}$ is self-adjoint. 
\hfill $\diamond$
\end{remark} 

As a special case we recall that $T_{min}$ is now characterized by 
\begin{align}
\begin{split} 
& T_{min} f = \tau f,    \\
& f \in \dom(T_{min})= \big\{g\in\dom(T_{max})  \, \big| \, \wti g(a) = {\wti g}^{\, \prime}(a) = \wti g(b) = {\wti g}^{\, \prime}(b) = 0\big\}.    \lb{2.16}
\end{split} 
\end{align}
Similarly, the Friedrichs extension $T_F$ of $T_{min}$ now permits the particularly simple characterization in terms of the generalized boundary values $\wti g(a), \wti g(b)$ as derived by Niessen and Zettl \cite{NZ92}(see also \cite{GP79}, \cite{Ka72}, \cite{Ka78}, \cite{KKZ86}, \cite{MZ00}, \cite{Re51}, \cite{Ro85}, \cite{YSZ15}), 
\begin{align}
T_F f = \tau f, \quad f \in \dom(T_F)= \big\{g\in\dom(T_{max})  \, \big| \, \wti g(a) = \wti g(b) = 0\big\}.    \lb{2.17}
\end{align}
Finally, we also mention the special case of the Krein--von Neumann extension in the case where 
$T_{min}$ is strictly positive, a result recently derived in \cite{FGKLNS21} (see also \cite{GLNPS21}).

\begin{theorem} \lb{t2.8}
In addition to Hypothesis \ref{h2.1}, suppose that $T_{min} \geq \varepsilon I$ for some 
$\varepsilon > 0$. Then the following items $(i)$ and $(ii)$ hold: \\[1mm]
$(i)$ Assume that $n_{\pm}(T_{min}) = 1$ and denote the principal solutions of $\tau u = 0$ at $a$ and $b$ by $u_a(0,\dott)$ and $u_b(0,\dott)$, respectively. If $\tau$ is in the limit circle case at $a$ and in the limit point case at $b$, then the Krein--von Neumann extension $T_{\a_K}$ of $T_{min}$ is given by 
\begin{align}
& T_{\a_K}f = \tau f,   \no \\
& f \in \dom(T_{\a_K})=\big\{g\in\dom(T_{max}) \, \big| \, \sin(\a_K) {\wti g}^{\, \prime}(a) 
+ \cos(\a_K) \wti g(a) = 0\big\},     \lb{2.18} \\
& \cot(\a_K) = - \wti u_b^{\, \prime}(0,a) / \wti u_b(0,a), \quad \a_K \in (0,\pi).   \no 
\end{align}
Similarly, if $\tau$ is in the limit circle case at $b$ and in the limit point case at $a$, then the Krein--von Neumann extension $T_{\b_K}$ of $T_{min}$ is given by 
\begin{align}
& T_{\b_K}f = \tau f,   \no \\
& f \in \dom(T_{\b_K})=\big\{g\in\dom(T_{max}) \, \big| \, \sin(\b_K) {\wti g}^{\, \prime}(b) 
+ \cos(\b_K) \wti g(b) = 0\big\},     \lb{2.19} \\
& \cot(\b_K) = - \wti u_a^{\, \prime}(0,b) / \wti u_a(0,b), \quad \b_K \in (0,\pi).   \no 
\end{align}
$(ii)$ Assume that $n_{\pm}(T_{min}) = 2$, that is, $\tau$ is in the limit circle case at $a$ and $b$. Then
the Krein--von Neumann extension $T_{0,R_K}$ of $T_{min}$ is given by 
\begin{align}
\begin{split} 
& T_{0,R_K} f = \tau f,    \\
& f \in \dom(T_{0,R_K})=\bigg\{g\in\dom(T_{max}) \, \bigg| \begin{pmatrix} \wti g(b) 
\\ {\wti g}^{\, \prime}(b) \end{pmatrix} = R_K \begin{pmatrix}
\wti g(a) \\ {\wti g}^{\, \prime}(a) \end{pmatrix} \bigg\},      \lb{2.21}
\end{split}
\end{align}
where 
\begin{align}
R_K=\begin{pmatrix} \wti{\hatt u}_{a}(0,b) & \wti u_{a}(0,b)\\
\wti{\hatt u}_{a}^{\, \prime}(0,b) & \wti u_{a}^{\, \prime}(0,b)
\end{pmatrix}.  \lb{2.22}
\end{align}
\end{theorem}

\medskip

In the remainder of this section we specialize the discussion to Bessel-type operators associated with 
differential expressions of the form 
\begin{align}
\begin{split} 
\tau_{s_a,s_b} = - \f{d^2}{dx^2} + \f{s_a^2 - (1/4)}{(x-a)^2} + \f{s_b^2 - (1/4)}{(x-b)^2} + q(x), \quad x \in (a,b),& \\
s_a, s_b \in [0,\infty), \; q \in L^{\infty}((a,b); dx), \; q \text{ real-valued~a.e.~on $(a,b)$,}&     \lb{2.23} 
\end{split}
\end{align}
where $(a,b) \subset \bbR$ is a bounded interval. The associated minimal and maximal operators are then denoted by 
$T_{s_a,s_b,min}$ and $T_{s_a,s_b,max}$ and the associated self-adjoint extensions of $T_{s_a,s_b,min}$, in obvious notation, are then abbreviated by $T_{s_a,s_b,\al,\be}$ and $T_{s_a,s_b,\varphi,R}$, respectively. 

As is well-known, $\tau_{s_a,s_b}$ is in the limit circle case at $a$ (resp., $b$) if and only if 
$s_a^2 \in [0,1)$ (resp., $s_b^2 \in [0,1)$) and in the limit point case at $a$ (resp., $b$) if and only if $s_a^2 \in [1,\infty)$ 
(resp., $s_b^2 \in [1,\infty)$). In addition, $\tau_{s_a,s_b}$ is nonoscillatory at $a$ (resp., $b$) if and only if 
$s_a^2 \in [0,\infty)$ (resp., $s_b^2 \in [0,\infty)$). For simplicity, we will always assume $s_a, s_b \in [0,\infty)$. 

Viewing $q$ as a bounded perturbation (which does not influence operator domains) one can focus on the case 
$q = 0$ and hence obtains (for $c,d \in (a,b)$ with $c$ sufficiently close to $a$ and $d$ sufficiently close to $b$), 
\begin{align}
u_{a,s_a,1/2}(0, x; q = 0) &= (x-a)^{(1/2) + s_a}, \; s_a \in [0,1), \; x \in (a,c),      \lb{2.24} \\
\hatt u_{a,s_a,1/2}(0, x; q = 0) &= \begin{cases} (2 s_a)^{-1} (x-a)^{(1/2) - s_a}, & s_a \in (0,1), \\
(x-a)^{1/2} \ln(1/(x-a)), & s_a =0, \end{cases} \quad x \in (a,c),     \lb{2.25} \\
u_{b,1/2,s_b}(0, x; q = 0) &= - (b-x)^{(1/2) + s_b}, \; s_b \in [0,1), \; x \in (d,b),    \lb{2.26} \\
\hatt u_{b,1/2,s_b}(0, x; q = 0) &= \begin{cases} (2 s_b)^{-1} (b-x)^{(1/2) - s_b}, & s_b \in (0,1), \\
(b-x)^{1/2} \ln(1/(b-x)), & s_b =0, \end{cases} \quad x \in (d,b).     \lb{2.27} 
\end{align}

Making the transition from $q = 0$ to $q \in L^{\infty}((a,b);dx)$, the principal and nonprincipal solutions for 
$\lambda = 0$ then have the same leading behavior near $x =a, b$ as in \eqref{2.24}--\eqref{2.27}, with easy control over the remainder terms, so that 
for $g \in \dom(T_{s_a,s_b,max})$, the Bessel operator boundary values then become 
\begin{align}
\begin{split} 
\wti g(a) &= - W(u_{a,s_a,1/2}(0, \dott), g)(a)   \\
&= \begin{cases} \lim_{x \downarrow a} g(x)\big/\big[(2 s_a)^{-1}(x-a)^{(1/2) - s_a}\big], & s_a \in (0,1), \\[1mm]
\lim_{x \downarrow a} g(x)\big/\big[(x-a)^{1/2} \ln(1/(x-a))\big], & s_a =0, 
\end{cases}    \lb{2.28} 
\end{split} \\
\begin{split} 
\wti g^{\, \prime} (a) &= W(\hatt u_{a,s_a,1/2}(0, \dott), g)(a)   \\
&= \begin{cases} \lim_{x \downarrow a} \big[g(x) - \wti g(a) (2 s_a)^{-1} (x-a)^{(1/2) - s_a}\big]\big/(x-a)^{(1/2) + s_a}, 
& s_a \in (0,1), \\[1mm]
\lim_{x \downarrow a} \big[g(x) - \wti g(a) (x-a)^{1/2} \ln(1/(x-a))\big]\big/(x-a)^{1/2}, & s_a =0, \\
\end{cases}
\end{split} \lb{2.29} \\
\begin{split} 
\wti g(b) &= - W(u_{b,1/2,s_b}(0, \dott), g)(b)   \\
&= \begin{cases} \lim_{x \uparrow b} g(x)\big/\big[(2 s_b)^{-1}(b-x)^{(1/2) - s_b}\big], & s_b \in (0,1), \\[1mm]
\lim_{x \uparrow b} g(x)\big/\big[(b-x)^{1/2} \ln(1/(b-x))\big], & s_b =0, 
\end{cases} 
\end{split} \lb{2.30} \\
\begin{split} 
\wti g^{\, \prime} (b) &= W(\hatt u_{b,1/2,s_b}(0, \dott), g)(b)   \\
&= \begin{cases} \lim_{x \uparrow b} \big[g(x) - \wti g(b) 
(2 s_b)^{-1} (b-x)^{(1/2) - s_b}\big]\big/\big[-(b-x)^{(1/2) + s_b}\big],\\
\hfill s_b \in (0,1), \\[1mm]
\lim_{x \uparrow b} \big[g(x) - \wti g(b) (b-x)^{1/2} \ln(1/(b-x))\big]\big/ \big[-(b-x)^{1/2}\big], \quad s_b =0.
\end{cases}    \lb{2.31} 
\end{split} 
\end{align}
Since $\tau_{s_a,s_b}$ is in the limit point case at $a$ (resp., $b$) if and only if $s_a \in [1,\infty)$ (resp., 
$s_b \in [1,\infty)$), there are no boundary values at $x=a$ (resp., $x=b$) unless $s_a \in [0,1)$ 
(resp., $s_b \in [0,1)$). 

All self-adjoint extensions of $T_{s_a,s_b,min,}$, $s_a, s_b \in [0,\infty)$, are then described as in 
Theorem \ref{t2.6} and Remark \ref{r2.7}, with special cases like the Friedrichs and Krein--von Neumann extension described as in \eqref{2.17} and Theorem \ref{t2.8}.

\section{Domain Properties of First-Order Singular Operators} \lb{s3}

In this section we develop our principal results on singular first-order differential operators associated with the differential expressions $\alpha_{s_a}$ (resp., $\beta_{s_b}$) and $\tau_{s_a,s_b}$.

Introducing the differential expressions $\alpha_{s_a}$, $\alpha^+_{s_a}$ by
\begin{align}
\begin{split} 
& \alpha_{s_a} = \f{d}{dx} - \f{s_a + (1/2)}{x-a} = (x-a)^{s_a + (1/2)} \f{d}{dx} (x-a)^{- s_a - (1/2)}, \\
& \alpha_{s_a}^+ = - \f{d}{dx} - \f{s_a + (1/2)}{x-a} = - \alpha_{-s_a - 1}, \quad 
s_a \in \bbR, \; x \in (a,b),      \lb{3.1}
\end{split} 
\end{align}
one confirms that 
\begin{equation}
\omega_{s_a} = \alpha_{s_a}^+ \alpha_{s_a}^{} = - \f{d^2}{dx^2} + \f{s_a^2 - (1/4)}{(x-a)^2}, \quad 
s_a \in \bbR, \; x \in (a,b).   \lb{3.2} 
\end{equation}

Occasionally, we will permit $s_a \in \bbC$. 

Next we introduce the preminimal and maximal operators associated with the differential expression 
$\alpha_{s_a}$, $s_a \in \bbR$ as follows:  
\begin{align}
& \dot A_{s_a,min} f = \alpha_{s_a} f, \quad s_a \in \bbR,      \no \\
& f \in\dom\big(\dot A_{s_a,min}\big) = \big\{g \in L^2((a,b); dx) \, \big| \, g \in AC_{loc}((a,b)); \, 
 \alpha_{s_a} g \in L^2((a,b); dx);  \no \\
&\hspace*{7.2cm} \supp(g) \subset (a,b) \text{ compact} \big\},    \lb{3.3}   \\
& A_{s_a,max} f = \alpha_{s_a} f, \quad s_a \in \bbR,         \lb{3.4} \\
& f \in\dom(A_{s_a,max}) = \big\{g \in L^2((a,b); dx) \, \big| \, g \in AC_{loc}((a,b)); \, 
 \alpha_{s_a} g \in L^2((a,b); dx)\big\}.   \no 
\end{align}
Since $(x-a)^{-1}$ is an $L^2_{loc}((a,b);dx)$ coefficient, one could have introduced, without loss of generality, the preminimal operator defined on the domain $C_0^{\infty}((a,b))$, that is, 
\begin{equation}
 \ddot A_{s_a,min} = \alpha_{s_a}\big|_{C_0^{\infty}((a,b))}, \quad s_a \in \bbR. 
\end{equation}

\begin{lemma} \lb{l3.1} 
Let $s_a \in \bbR$, then 
\begin{equation}
\big(\dot A_{s_a,min}\big)^* = - A_{-s_a - 1,max}, \quad A_{-s_a - 1,max}^* = - \ol{\dot A_{s_a,min}}.   \lb{3.6}
\end{equation}
In particular, $A_{s_a,max}$ is closed in $L^2((a,b); dx)$ and hence $ \dot A_{s_a,min}$ is closable in $L^2((a,b); dx)$ for all $s_a \in \bbR$.
\end{lemma}
\begin{proof}
An integration by parts assuming $f \in \dom\big(\dot A_{s_a,min}\big)$, $g \in \dom(A_{-s_a - 1,max})$ immediately proves that 
\begin{equation}
\big(g, \dot A_{s_a,min} f\big)_{L^2((a,b);dx)} = - (A_{-s_a - 1,max} g,f)_{L^2((a,b);dx)},
\end{equation}
and hence $\big(\dot A_{s_a,min}\big)^* \supseteq - A_{-s_a - 1,max}$, $s_a \in \bbR$. To prove the converse inclusion one follows the usual strategy: Suppose $g \in \dom\big(\big(\dot A_{s_a,min}\big)^*\big)$, introduce  
$h_{s_a} = (\dot A_{s_a,min}\big)^* g$, and let $f \in \dom\big(\dot A_{s_a,min}\big)$, $s_a \in \bbR$. Then
\begin{align}
\begin{split} 
& \big(\big(\dot A_{s_a,min}\big)^* g,f\big)_{L^2((a,b);dx)} = \int_a^b dx \, \ol{h_{s_a}(x)} f(x) = 
\big(g, \dot A_{s_a,min} f\big)_{L^2((a,b);dx)}      \\
& \quad = \int_a^b dx \, \big\{\ol{g(x)} f'(x) - [s_a +(1/2)] x^{-1} \ol{g(x)} f(x)\big\},
\end{split} 
\end{align}
that is, introducing
\begin{equation}
H_{s_a}(x) = \int_x^c dt \, \big\{h_{s_a}(t) + [s_a +(1/2)] t^{-1} g(t)\big\}, \quad s_a \in \bbR, \; x \in (a,b],   
\end{equation}
for some $c \in (a,b)$, one obtains $H_{s_a} \in AC_{loc}((a,b))$ and 
\begin{align}
\begin{split} 
& \int_a^b dx \, \ol{g(x)} f'(x) = \int_a^b  dx \, \big\{\ol{h_{s_a}(x)} + [s_a +(1/2)] x^{-1} \ol{g(x)}\big\} f(x)  \\
& \quad = - \int_a^b dx \, \ol{H_{s_a}'(x)} f(x) = \int_a^b dx \, \ol{H_{s_a}(x)} f'(x), 
\end{split}
\end{align}
and hence
\begin{equation}
\int_a^b  dx \, \ol{[g(x) - H_{s_a}(x)]} f'(x) = 0, \quad f \in \dom\big(\dot A_{s_a,min}\big).   \lb{3.11} 
\end{equation}
In particular, \eqref{3.11} holds for all $f \in C_0^{\infty}((a,b))$ and thus,
\begin{equation}
[g(x) - H_{s_a}(x)] = c
\end{equation}
for some $c \in \bbC$. Therefore, $g \in AC_{loc}((a,b))$ and 
\begin{equation}
0 = g'(x) - H_{s_a}'(x) = g'(x) + h_{s_a}(x) + [s_a+(1/2)] x^{-1} g(x), \quad x \in (a,b), 
\end{equation}
implying 
\begin{equation}
h_{s_a} = - \alpha_{-s_a-1} g \in L^2((a,b);dx),
\end{equation}
and hence $g \in \dom(A_{-s_a - 1,max})$ and $\big(\dot A_{s_a,min}\big)^* \subseteq - A_{-s_a - 1,max}$ yields that  
$\big(\dot A_{s_a,min}\big)^* = - A_{-s_a - 1,max}$, $s_a \in \bbR$. In particular, $A_{s_a,max}$ is closed for all 
$s_a \in \bbR$ and thus $\dot A_{s_a,min} \subset A_{s_a,max}$ implies that  $\dot A_{s_a,min}$ is closable for all 
$s_a \in \bbR$. Taking adjoints in the first relation of \eqref{3.6} then yields the second relation in \eqref{3.6}.  
\end{proof}

One then naturally defines the minimal operator associated with $\alpha_{s_a}$ via the closure of 
$\dot A_{s_a,min}$, 
\begin{equation}
A_{s_a,min} = \ol{\dot A_{s_a,min}} = \ol{\alpha_{s_a}\big|_{C_0^{\infty}((a,b))}}, \quad s_a \in \bbR, 
\end{equation}
and obtains
\begin{equation}
A_{s_a,min}^* = - A_{-s_a - 1,max}, \quad A_{-s_a - 1,max}^* = - A_{s_a,min}, \quad s_a \in \bbR.   \lb{3.16}
\end{equation}

The principal result regarding $A_{s_a,min}$ and $A_{s_a,max}$ then reads as follows.

\begin{theorem} \lb{t3.2}
Let $s_a \in \bbR$. \\[1mm]
$(i)$ For all $s_a \in \bbR \backslash \{0\}$, 
\begin{equation}
\dom(A_{s_a,min}) = H_0^1((a,b)).    \lb{3.17}
\end{equation}
$(ii)$ For $s_a = 0$, $C_0^{\infty}((a,b))$, and hence $H^1_0((a,b))$, is a core for $A_{0,min}$, in addition, 
$H_0^1((a,b)) \subsetneqq \dom(A_{0,min})$. Moreover,  
\begin{align}
\dom(A_{0,min}) &= \big\{f \in L^2((a,b); dx) \, \big| \, f \in AC_{loc}((a,b)); \, f(a) = 0 = f(b);   \no \\
& \hspace*{5.8cm} \alpha_0 f \in L^2((a,b); dx)\big\},    \lb{3.18} \\
\dom(A_{0,max}) &= \big\{f \in L^2((a,b); dx) \, \big| \, f \in AC_{loc}((a,b)); \, \alpha_0 f \in L^2((a,b); dx)\big\}   \no \\
&= \big\{f \in L^2((a,b); dx) \, \big| \, f \in AC_{loc}((a,b)); \, f(a) = 0;     \lb{3.19} \\
& \hspace*{4.7cm}  \alpha_0 f \in L^2((a,b); dx)\big\}.   \no 
\end{align} 
In fact, the boundary condition $f(a)=0$ in \eqref{3.18} and \eqref{3.19} can be replaced by 
\begin{equation} 
\lim_{x \downarrow a} \f{f(x)}{\big[(x-a) \ln(R/(x-a))]^{1/2}} = 0,     \lb{3.20}
\end{equation}
and the boundary condition $f(b)=0$ in \eqref{3.18} can be replaced by 
\begin{equation}
\lim_{x \uparrow b} \f{f(x)}{(b-x)^{1/2}}= 0.      \lb{3.21}
\end{equation} 
$(iii)$ For $s_a \in (-\infty, -1] \cup (0,\infty)$ one has, 
\begin{align}
\dom(A_{s_a,max}) &= \big\{f \in L^2((a,b); dx) \, \big| \, f \in AC_{loc}((a,b)); \, 
\alpha_{s_a} f \in L^2((a,b); dx)\big\}    \no \\
& = \big\{f \in L^2((a,b); dx) \, \big| \, f \in AC_{loc}((a,b)); \, f(a) = 0;    \lb{3.22} \\
& \hspace*{4.6cm} \alpha_{s_a} f \in L^2((a,b); dx)\big\}.   \no
\end{align}
The boundary condition $f(a)=0$ in \eqref{3.22} can be replaced by 
\begin{equation}
\lim_{x \downarrow a} \f{f(x)}{(x-a)^{1/2}} = 0.      \lb{3.23}
\end{equation}
$(iv)$ For $s_a \in (-1,0)$, 
\begin{equation}
H_0^1((a,b)) = \dom(A_{s_a,min}) \subsetneqq \dom(A_{s_a,max}).     \lb{3.24} 
\end{equation}
\end{theorem}
\begin{proof}
Employing locality of $A_{s_a,min}, A_{s_a,max}$ and the fact that $(x-a)^{-1}$ is bounded on $[a+\varepsilon, b]$ for all 
$0 < \varepsilon < (b-a)$, all assertions in a neighborhood of $x=b$ follow immediately from the special and well-known case $s_a = - 1/2$ as $[s_a +(1/2)] (x-a)^{-1} \chi_{[a+\varepsilon,b]}$ represents a bounded operator of multiplication in $L^2((a,b); dx) $. In particular, $f \in \dom(A_{-1/2,min})$ behaves in a neighborhood of $x=b$ like an 
$H^1_0((a,b))$-element and hence $\lim_{x \uparrow b} f(x)\big/(b-x)^{1/2} = 0$ holds (see \eqref{B.10b}) implying 
\eqref{3.21}. Similarly, $f \in \dom(A_{-1/2,max})$ behaves in a neighborhood of $x=b$ like an $H^1((a,b))$-element. 
Thus, it suffices to focus entirely on the singular left endpoint $x=a$ in the remainder of this proof. \\[1mm] 
\noindent 
$(1)$ $s_a \in (0,\infty)$: Then the relations in \eqref{B.6} prove \eqref{3.17} and \eqref{3.22} in this case. \\[1mm] 
\noindent 
$(2)$ $s_a \in (-1,0)$: In this case the Hardy inequality \eqref{B.11} shows that the operator of multiplication 
$\pm [s_a + (1/2)]/x$ is bounded relatively to $\alpha_{-1/2}\big|_{H^1_0((a,b))}$ with bound strictly less than one, proving 
\eqref{3.17} in the case $s_a \in (-1,0)$. \\[1mm] 
$(3)$ $s_a \in (-\infty,-1]$: In this case we mimic the proof of Lemma \ref{lB.2} and obtain  
\begin{align}
& \big|(x-a)^{-s_a -(1/2)} f(x) - (c-a)^{-s_a - (1/2)} f(c)\big| = \bigg|\int_c^x dt \, \big[(t-a)^{-s_a - (1/2)} f(t)\big]' \bigg|    \no \\
& \quad = \bigg|\int_c^x dt \, (t-a)^{-s_a - (1/2)} (\alpha_{s_a} f)(t) \bigg|    \no \\
& \quad \leq 
\bigg(\int_c^x dt \, (t-a)^{-2s_a - 1}\bigg)^{1/2} \|\alpha_{s_a} f\|_{L^2((c,x);dt)}    \no \\
& \quad = (-2s_a)^{-1/2} \big[(x-a)^{-2s_a} - (c-a)^{-2s_a}\big]^{1/2}  \|\alpha_{s_a} f\|_{L^2((c,x);dt)},    \lb{3.25} \\
& \hspace*{4.35cm} c\in(a,b), \; x \in [c,b), \; s_a \in (-\infty, -1].    \no
\end{align}
Since by hypothesis $\alpha_{s_a} f \in L^2((a,b); dt)$ and $s_a < 0$, one concludes from lines two and three in \eqref{3.25} 
the existence of the limit
\begin{equation}
F_{s_a} := \lim_{c \downarrow a} (c-a)^{-s_a - (1/2)} f(c), \quad s_a \in (-\infty,-1],     \lb{3.26} 
\end{equation}
and hence obtains
\begin{align}
\begin{split} 
& \bigg|\f{f(x)}{(x-a)^{1/2}} - F_{s_a} (x-a)^{s_a}\bigg| \leq (-2s_a)^{1/2} \|\alpha_{s_a} f\|_{L^2((a,x);dt)}   \\
& \quad \leq (-2s_a)^{1/2} \|\alpha_{s_a} f\|_{L^2((a,b);dt)}, \quad x \in (a,b), \; s_a \in (-\infty, -1].    \lb{3.27} 
\end{split}
\end{align}
Since $f \in L^2((a,b);dt)$, that is, $\int_a^b dt \, (t-a)^{-1} \big|(t-a)^{1/2} f(t)\big|^2 < \infty$ implies 
\begin{equation} 
\liminf_{x \downarrow a} (x-a)^{1/2} |f(x)| =0, 
\end{equation} 
the existence of $F_{s_a}$ in \eqref{3.26} yields 
\begin{equation}
F_{s_a} = 0, \quad s_a \in (-\infty, -1]. 
\end{equation}
Hence, by the first line in \eqref{3.27},  
\begin{equation}
\lim_{x \downarrow a} \f{f(x)}{(x-a)^{1/2}} = 0, \quad f \in \dom(A_{s_a,max}), \; s_a \in (-\infty, -1].     \lb{3.30} 
\end{equation}
This proves \eqref{3.22}, and \eqref{3.23} also for the case $s_a \in (-\infty,-1]$. 

Next, \eqref{B.1} implies $\int_a^b dx \, (x-a)^{-2} |f(x)|^2 < \infty$, and hence $\alpha_{s_a} f \in L^2((a,b);dx)$ 
and \eqref{B.4} also yields $f' \in L^2((a,b);dx)$, and therefore $f \in H_0^1((a,b))$. This proves \eqref{3.17} also 
for the case $s_a \in (-\infty,-1]$ and hence completes the proof of items $(i)$ and $(iii)$. \\[1mm]
\noindent 
$(4)$ $s_a=0$: First one notes that \eqref{B.5} applies and hence 
\begin{equation}
f \in \dom(A_{0,max}) \, \text{ implies } \, f(a)= 0,  
\end{equation}
in fact, it even implies $\lim_{x \downarrow a} |f(x)|\big/[(x-a) \ln(R/(x-a))\big]^{1/2} = 0$ and hence \eqref{3.20}. By 
\cite[Proposition~3.1\,$(i)$,\,$(iii)$]{BDG11}, it is known that that in the case $b=\infty$, one has 
\begin{align}
& \dom(A_{0,min,b=\infty}) = \dom(A_{0,max,b=\infty})      \no \\
& \quad = \big\{f \in L^2((a,\infty);dx) \big| f \in AC_{loc}((a,\infty)); \, f(a) = 0; \, \alpha_0 f \in L^2((a,\infty);dx)\big\}  \no \\
& \quad  = \big\{f \in L^2((a,\infty);dx) \big| f \in AC_{loc}((a,\infty)); \, \, \alpha_0 f \in L^2((a,\infty);dx)\big\}.  \lb{3.32}
\end{align}
By locality of $\alpha_0$, the local properties of elements in \eqref{3.32} and in $\dom(A_{0,min})$ and 
$\dom(A_{0,max})$ in \eqref{3.18} and \eqref{3.19}, respectively, coincide on any interval $(a,b-\varepsilon)$, 
$0 < \varepsilon < (b-a)$. The point $b$ is of course rather different for $\dom(A_{0,min})$ and $\dom(A_{0,max})$ in the sense that 
\begin{equation}
f \in \dom(A_{0,min}) \, \text{ necessitates the Dirichlet boundary condition} \, f(b) = 0,
\end{equation}
whereas $g \in \dom(A_{0,max})$ enforces no boundary condition at all at $x=b$. This is clear from the fact that 
$(x-a)^{-1} \chi_{[a+\varepsilon,b]}(x)$, $x \in (a,b)$, $0 < \varepsilon < (b-a)$, generates a bounded operator of multiplication in $L^2((a,b); dx)$ and clearly,
\begin{align}
\dom(A_{-1/2,min}) = H^1_0((a,b)), \quad  \dom(A_{-1/2,max}) = H^1((a,b)).
\end{align} 
In particular, there necessarily is a Dirichlet boundary condition for each element in $\dom(A_{-1/2,min})$, and hence 
in $\dom(A_{0,min})$, at $x=b$, but no boundary condition for elements in $\dom(A_{-1/2,max})$, and hence in 
$\dom(A_{0,max})$, at $x=b$.  
Thus, \eqref{3.18} and the second equality in \eqref{3.19} holds (the first equality in \eqref{3.19} holds by \eqref{3.4}). 

To prove that $C_0^{\infty}((a,b))$ (and hence $H^1_0((a,b))$) is a core for $A_{0,min}$, it only remains to show that for all $f\in\dom(\dot A_{0,min})$, there exists a sequence $\{f_n\}_{n\in\bbN}\subseteq C_0^{\infty}((a,b))$ such that 
$\norm{f_n-f}_{L^2((a,b);dx)}\rightarrow 0$ and $\norm{\alpha_0 f_n-\alpha_0 f}_{L^2((a,b);dx)}\rightarrow 0$ as $n\rightarrow\infty$. In this context, let $h:\bbR\rightarrow [0,\infty)$ be a $C^\infty$-function satisfying
\begin{align}
\begin{cases}
h \text{ is even on } \bbR, \\
h(x)\geq0,\quad x\in\bbR,\\
\supp(h)\subseteq(-1,1),  \\[1mm] 
\int_{-1}^1 dx \, h(x)=1, \\[1mm] 
h \text{ is non-increasing on } [0,\infty).
\end{cases}
\end{align}
In addition, for $n\in\bbN$, let $h_n(x)=n h(nx)$ and $f_n=f*h_n\in C^\infty(\bbR)$.
Letting $\supp(f)\subseteq[a',b']$ where $a<a'<b'<b$, for $n\in\bbN$ sufficiently large, that is, for
\begin{equation}
1/n<(1/2)\min(a'-a,b-b'),   \lb{3.36}
\end{equation}
one has
\begin{equation}
\supp(f_n)\subseteq([a+a']/2,[b+b']/2)\subset (a,b).
\end{equation}
Therefore $f_n\in C_0^{\infty}(([a+a']/2,[b+b']/2))$ for all $n\in\bbN$ satisfying \eqref{3.36}. Thus, the singularity of $\alpha_0$ at $x=a$ has no significance on $f$ or $f_n,\;  n\in\bbN$, satisfying \eqref{3.36}. By the standard theory of convolution (see, e.g., \cite[Sects.~7.2,~7.3]{GT01}), $f_n \in C^{\infty}_0((a, b)),$
$f_n\rightarrow f$, and $f_n'=h_n*f'\rightarrow f'$ as $n\rightarrow\infty$ in $L^2(([a+a']/2,[b+b']/2);dx)$. Hence, $\norm{f_n-f}_{L^2((a,b);dx)}\rightarrow0$ and $\norm{\alpha_0 f_n-\alpha_0 f}_{L^2((a,b);dx)}\rightarrow0$ as $n\rightarrow\infty$, showing $C_0^\infty((a,b))$ is a core for $A_{0,min}$.

Since 
\begin{equation} 
f_0(x) = (x-a)^{1/2} \wti \chi_{[a,a+\varepsilon]}(x), \quad x \in (a,b), \; \varepsilon \in (0,(b-a)/2),
\end{equation} 
satisfies $f_0 \in \dom(A_{0,min})$, but $f_0 \notin H_0^1((a,b))$, this completes the proof of item $(ii)$. \\[1mm] 
\noindent 
$(5)$ Finally, considering 
\begin{equation} 
f_{s_a}(x) = (x-a)^{s_a + (1/2)} \wti \chi_{[a,a+\varepsilon]}(x), \quad s_a \in (-1,0), \; x \in (a,b), 
\; \varepsilon \in (0,(b-a)/2),
\end{equation} 
one readily verifies that 
\begin{equation}  
f_{s_a} \in \dom(A_{s_a,max}), \quad f_{s_a} \notin H_0^1((a,b)) = \dom(A_{s_a,min}), \quad s_a \in (-1,0),
\end{equation}
proving item $(iv)$.
\end{proof}

We emphasize that more can and has been proven in this context in \cite{BDG11} (see also \cite{AA12}, \cite{AB16}, \cite{AB15}, \cite{DF21}, \cite{DG21}, \cite{KT11}) using quite different methods, not involving Hardy-type inequalities such as \eqref{B.1}. In particular, the case of complex-valued $s_a$ is considered in \cite{BDG11}, \cite{DF21}, \cite{DG21}. 

\begin{remark} \lb{r3.3}
One might ask by how much $\dom(A_{0,min})$ misses out on coinciding with $H^1_0((a,b))$. In fact, not by much as the following elementary consideration shows. It suffices to treat $A_{0,max}$ and focus on the behavior of functions in its domain near the left endpoint $x=a$ only. Suppose $0 \neq f \in \dom(A_{0,max})$ and denote $h = A_{0,max} f$. Then the first-order differential equation $(\alpha_0 f)(x) = h(x)$ for $x \in (a,a+1)$, say (i.e., assuming $b > a+1$ without loss of generality), implies
\begin{equation}
f(x) = C (x-a)^{1/2} - (x-a)^{1/2} \int_x^{a+1} dt \, (t-a)^{-1/2} h(t), \quad x \in (a,a+1), 
\end{equation}
for some $C \in \bbC$. Thus a Cauchy estimate yields
\begin{equation}
|f(x)| \leq |C| (x-a)^{1/2} + (x-a)^{1/2} [\ln(1/(x-a))]^{1/2} \|h\|_{L^2((a,a+1); dt)},
\end{equation}
hence $f(a)=0$ and obviously $f \in L^2((a,a+1);dx)$. In addition,
\begin{align}
& f'(x) = 2^{-1} C (x-a)^{-1/2} - 2^{-1} (x-a)^{-1/2} \int_x^{a+1} dt \, (t-a)^{-1/2} h(t) + h(x),       \no \\
& \hspace*{8.5cm} x \in (a,a+1), 
\end{align}
and the same Cauchy estimate implies
\begin{align}
\begin{split} 
|f'(x)| & \leq |C/2| (x-a)^{-1/2} +  (x-a)^{-1/2} [\ln(1/(x-a))]^{1/2} \|h\|_{L^2((a,a+1);dt)}   \\
& \quad + |h(x)|,  \quad x \in (a,a+1). 
\end{split} 
\end{align}
In particular, the possible failure of $f'$ being $L^2$ near $x=a$ happens only on a logarithmic scale as 
\begin{equation}
\int_{a+\varepsilon}^{a+1} dx \, |f'(x)|^2 \leq 2^{-1} [\ln(1/\varepsilon)]^2 \|h\|_{L^2((a,a+1);dx)}^2 
\big[1 + \Oh\big((\ln(1/\varepsilon))^{-1/2}\big)\big] 
\end{equation}
for $0 < \varepsilon$ sufficiently small. 

These observations extend of course to $s_a \in i \bbR$, that is, $\Re(s_a) = 0$. 
\hfill $\diamond$ 
\end{remark}

Actually, combining the idea behind Remark \ref{r3.3} with appropriate two-weight Hardy-type inequalities 
yields an interesting alternative proof of 
\begin{equation}
\dom(A_{s_a,max}) = H^1_0((a,b)), \quad s_a \in (-\infty, -1] \cup (0,\infty),       \lb{3.42} 
\end{equation} 
compared to that in Theorem \ref{t3.2} as we will indicate next in Lemma \ref{l3.4}.

We start by recalling the following pair of two-weight Hardy-type inequalities due to Muckenhoupt \cite{Mu72}, 
Talenti \cite{Ta69}, Tomaselli \cite{To69}, and Chisholm, Everitt, Littlejohn \cite{CE70/71}, \cite{CEL99} (see, also  \cite[p.~38--40]{KMP07} and the references therein): 
Let $p \in [1,\infty)$, $p^{-1} + (p')^{-1} = 1$, and $0 \leq f$ a.e. on $(a,b)$ be measurable, 
$b \in (a,\infty) \cup \{\infty\}$. Then 
\begin{equation}
\int_a^b dx \, u(x)\bigg(\int_a^x dt \, f(t)\bigg)^p \leq C^p \int_a^b dx \, v(x) |f(x)|^p      \lb{3.43} 
\end{equation} 
holds for some $C \in (0,\infty)$ if and only if
\begin{equation}
A = \sup_{c \in (a,b)} \bigg(\int_c^b ds \, u(s)\bigg)^{1/p} \bigg(\int_a^c dt \, v(t)^{1-p'}\bigg)^{1/p'} 
< \infty.
\end{equation}
In this case the smallest $C$ in \eqref{3.43} satisfies
\begin{align}
\begin{split} 
& A \leq C \leq p^{1/p} (p')^{1/p'} A, \quad p \in (1,\infty),   \\
& C = A, \quad p=1. 
\end{split} 
\end{align} 
Similarly, 
\begin{equation}
\int_a^b dx \, u(x)\bigg(\int_x^b dt \, f(t)\bigg)^p \leq D^p \int_a^bdx \,  v(x) |f(x)|^p      \lb{3.46} 
\end{equation} 
holds for some $D \in (0,\infty)$ if and only if
\begin{equation}
B = \sup_{c \in (a,b)} \bigg(\int_a^c ds \, u(s)\bigg)^{1/p} \bigg(\int_c^b dt \, v(t)^{1-p'}\bigg)^{1/p'} 
< \infty.
\end{equation}
In this case the smallest $D$ in \eqref{3.46} satisfies
\begin{align}
\begin{split} 
& B \leq D \leq p^{1/p} (p')^{1/p'} B, \quad p \in (1,\infty),   \\
& D = B, \quad p=1. 
\end{split} 
\end{align} 

\begin{lemma} \lb{l3.4}
Let $\Re(s_a) \in (-\infty,-1] \cup (0,\infty)$, then
\begin{align}
\dom(A_{s_a,max}) = H^1_0((a,b)).     \lb{3.50} 
\end{align}
\end{lemma}
\begin{proof} 
Consider $0 \neq f \in \dom(A_{s_a,max})$ and denote $h = A_{s_a,max} f$, $s_a \in \bbC$. \\[1mm] 
$(i)$ Suppose $\Re(s_a) \in (0,\infty)$. Then the first-order differential equation $(\alpha_{s_a} f)(x) = h(x)$ for $x \in (a,a+1)$, say, implies
\begin{align}
\begin{split}
f(x) = C_1 (x-a)^{s_a + (1/2)} - (x-a)^{s_a + (1/2)} \int_x^{a+1} dt \, (t-a)^{-s_a - (1/2)} h(t),& \\
x \in (a,a+1),&
\end{split} 
\end{align}
for some $C_1 \in \bbC$, and hence
\begin{align}
f'(x) &= C_1 [s_a + (1/2)] (x-a)^{s_a - (1/2)}     \no \\ 
& \quad - [s_a + (1/2)] (x-a)^{s_a - (1/2)} \int_x^{a+1} dt \, (t-a)^{-s_a - (1/2)} h(t) 
+ h(x),     \lb{3.51}  \\
& \hspace*{8.25cm} x \in (a,a+1).    \no
\end{align}
Since 
\begin{align}
& \bigg|(x-a)^{s_a + (1/2)} \int_x^{a+1} dt \, (t-a)^{-s_a - (1/2)} h(t)\bigg|     \no \\
& \quad \leq (x-a)^{s_a + (1/2)} \bigg(\int_x^{a+1} dt \, (t-a)^{-2s_a -1}\bigg)^{1/2} \|h\|_{L^2((a,a+1);dt)}    \\
& \quad = (2s_a)^{-1/2} (x-a)^{1/2} \big[1 - (x-a)^{2s_a}\big]^{1/2} \|h\|_{L^2((a,a+1);dt)}, \quad x \in (a,a+1),   \no
\end{align}
one obtains
\begin{equation}
|f(x)| \underset{x \downarrow a}{=} \Oh\big((x-a)^{1/2}\big) 
\end{equation}
(this is not optimal, see \eqref{B.6}, but it suffices for our current purpose).  
As $(x-a)^{s_a - (1/2)}$, $\Re(s_a) \in (0,\infty)$, generates an element in $L^2((a,a+1);dx)$, one next estimates 
\begin{align}
& \int_a^{a+1} dx \, (x-a)^{2s_a-1} \bigg|\int_x^{a+1} dt \, (t-a)^{-s_a - (1/2)} h(t)\bigg|^2     \\
& \quad \leq 
\int_a^{a+1} dx \, (x-a)^{2s_a - 1} \bigg(\int_x^{a+1} dt \, (t-a)^{-s_a - (1/2)} |h(t)|\bigg)^2 \leq d \int_a^{a+1} dx |h(x)|^2 \no
\end{align}
for some constant $d \in (0,\infty)$, upon choosing $u(x) = (x-a)^{2s_a - 1}$, $v(x) = (x-a)^{2s_a + 1}$, $p=p'=2$ in \eqref{3.46}. Thus, \eqref{3.51} implies 
\begin{equation}
f' \in L^2((a,a+1);dx) \, \text{ and hence } \, f \wti \chi_{[a,a+\varepsilon]} \in H^1_0((a,b)), \quad \Re(s_a) \in (0,\infty). 
\end{equation}
Since by Hardy's inequality \eqref{B.11}, $f \in H^1_0((a,b))$ implies $\alpha_{s_a} f \in L^2((a,b); dx)$ for all 
$s_a \in \bbC$, it follows that 
\begin{equation}
\dom(A_{s_a,max}) = H^1_0((a,b)), \quad \Re(s_a) \in (0,\infty).     \lb{3.56} 
\end{equation}
$(ii)$ Suppose $\Re(s_a) \in (-\infty, -1]$. Then again $(\alpha_{s_a} f)(x) = h(x)$ for $x \in (a,a+1)$, say, implies
\begin{equation}
f(x) = C_2 (x-a)^{s_a + (1/2)} + (x-a)^{s_a + (1/2)} \int_a^x dt \, (t-a)^{-s_a - (1/2)} h(t), \quad x \in (a,a+1), 
\end{equation}
for some $C_2 \in \bbC$, and hence
\begin{align}
f'(x) &= C_2 [s_a + (1/2)] (x-a)^{s_a - (1/2)}       \lb{3.58}  \\
& \quad + [s_a + (1/2)] (x-a)^{s_a - (1/2)} \int_a^x dt \, t^{-s_a - (1/2)} h(t) + h(x), \quad x \in (a,a+1).   \no 
\end{align}
Clearly, $f \in L^2((a,a+1);dx)$ requires the choice $C_2=0$. Since 
\begin{align}
\begin{split} 
& \bigg|(x-a)^{s_a + (1/2)} \int_a^x dt \, (t-a)^{-s_a - (1/2)} h(t)\bigg|   \\
& \quad \leq (x-a)^{1/2} \|h\|_{L^2((a,x);dt)}, \quad x \in (a,a+1), 
\end{split}
\end{align}
one obtains in particular, 
\begin{equation}
|f(x)| \underset{x \downarrow a}{=} \oh\big((x-a)^{1/2}\big).  
\end{equation} 
Estimating 
\begin{align}
& \int_a^{a+1} dx \, (x-a)^{2s_a - 1} \bigg|\int_a^x dt \, (t-a)^{-s_a - (1/2)} h(t)\bigg|^2      \lb{3.61} \\
& \quad \leq 
\int_a^{a+1} dx \, (x-a)^{2s_a - 1} \bigg(\int_a^x dt \, (t-a)^{-s_a - (1/2)} |h(t)|\bigg)^2  \leq c \int_a^{a+1} dx |h(x)|^2      
\no 
\end{align}
for some constant $c \in (0,\infty)$ upon choosing $u(x) = (x-a)^{2s_a - 1}$, $v(x) = (x-a)^{2s_a + 1}$, $p=p'=2$ in \eqref{3.43}, \eqref{3.58} implies 
\begin{equation}
f' \in L^2((a,a+1);dx) \, \text{ and hence } \, f \wti \chi_{[a,a+\varepsilon]} \in H^1_0((a,b)), \quad \Re(s_a) \in (-\infty,-1].     \lb{3.62} 
\end{equation}
This in turn again implies that 
\begin{equation}
\dom(A_{s_a,max}) = H^1_0((a,b)), \quad \Re(s_a) \in (-\infty,-1].       \lb{3.63} 
\end{equation}
\end{proof}

Of course, as shown in Theorem \ref{t3.2}, even $\dom(A_{s_a,min}) = \dom(A_{s_a,max}) = H^1_0((a,b))$ holds in \eqref{3.42} and \eqref{3.50}.

Similarly, one obtains the following alternative proof of Theorem \ref{t3.2} in the case $\Re(s_a) \in (-1,0)$.

\begin{lemma} \lb{l3.5}
Let $\Re (s_a) \in (-1,0)$, then
\begin{equation}
\dom(A_{s_a,min}) = H^1_0((a,b)).    \lb{3.64}
\end{equation}
\end{lemma}
\begin{proof} 
Consider $0 \neq f \in \dom(A_{s_a,min})$ and denote $h = A_{s_a,min} f$, $\Re(s_a) \in (-1,0)$. Then 
$(\alpha_{s_a} f)(x) = h(x)$ for $x \in (a,a+1)$, say, once again implies
\begin{equation}
f(x) = C_0 (x-a)^{s_a + (1/2)} + (x-a)^{s_a + (1/2)} \int_a^x dt \, t^{-s_a - (1/2)} h(t), \quad x \in (a,a+1),     \lb{3.65}
\end{equation}
for some $C_0 \in \bbC$, and hence
\begin{align}
f'(x) &= C_0 [s_a + (1/2)] (x-a)^{s_a - (1/2)}  \lb{3.66}   \\
& \quad + [s_a + (1/2)] (x-a)^{s_a - (1/2)} \int_a^x dt \, t^{-s_a - (1/2)} h(t) + h(x), \quad x \in (a,a+1).   \no 
\end{align}
Clearly $f \in L^2((a,a+1);dx)$, and in addition, since $f \in \dom(A_{s_a,min})$ necessitates 
\begin{equation}
\wti f(a) = \lim_{x \downarrow a} \f{f(x)}{x^{s_a + (1/2)}} = 0, 
\end{equation} 
(cf.\ \eqref{2.16}, \eqref{2.28}, and the first equality in \eqref{4.10}) one concludes that $C_0 = 0$ in \eqref{3.65}. Thus, once more \eqref{3.61} holds for some constant $c \in (0,\infty)$ upon choosing $u(x) = (x-a)^{2s_a-1}$, $v(x) = (x-a)^{2s_a+1}$, $p=p'=2$ in \eqref{3.43}. In particular, \eqref{3.62} and \eqref{3.64} hold. 
\end{proof}

Together, Remark \ref{r3.3} and Lemmas \ref{l3.4} and \ref{l3.5} recover (parts of) a finite interval analog of 
\cite[Proposition~3.1\,$(i)$,\,$(ii)$]{BDG11}. 

Next, one notes that if the singularity is not at $x=a$ but at $x=b$, that is, if $\alpha_{s_a}$, $\alpha^+_{s_a}$ are replaced by $\beta_{s_b}$, $\beta^+_{s_b}$, where  
\begin{align}
\begin{split} 
& \beta_{s_b} = \f{d}{dx} + \f{s_b + (1/2)}{b-x} = (b-x)^{s_b + (1/2)} \f{d}{dx} (b-x)^{- s_b - (1/2)}, \\
& \beta_{s_b}^+ = - \f{d}{dx} + \f{s_b + (1/2)}{b-x} = - \beta_{-s_b-1}, \quad 
s_b \in \bbR, \; x \in (a,b),      \lb{3.68}
\end{split} 
\end{align}
one obtains  
\begin{equation}
\eta_{s_b} = \beta_{s_b}^+ \beta_{s_b}^{} = - \f{d^2}{dx^2} + \f{s_b^2 - (1/4)}{(b-x)^2}, \quad s_b \in \bbR, \; x \in (a,b).   \lb{3.69} 
\end{equation}
Thus, with the roles of $x=a$ and $x=b$ interchanged, one replaces $\dot A_{s_a,min}$, $\ddot A_{s_a,min}$, 
$A_{s_a,min}$, $A_{s_a,max}$, by 
\begin{align}
& \dot B_{s_b,min} f = \beta_{s_b} f, \quad s_b \in \bbR,      \no \\
& f \in\dom\big(\dot B_{s_b,min}\big) = \big\{g \in L^2((a,b); dx) \, \big| \, g \in AC_{loc}((a,b)); \, 
 \beta_{s_b} g \in L^2((a,b); dx);  \no \\
&\hspace*{6.8cm} \supp(g) \subset (a,b) \text{ compact} \big\},       \\
& \ddot B_{s_b,min} = \beta_{s_b}\big|_{C_0^{\infty}((a,b))}, \quad s_b \in \bbR,   \\
& B_{s_b,min} = \ol{\dot B_{s_b,min}} = \ol{\beta_{s_b}\big|_{C_0^{\infty}((a,b))}}, \quad s_b \in \bbR,    \\
& B_{s_b,max} f = \beta_{s_b} f, \quad s_b \in \bbR,      \\
& f \in\dom(B_{s_b,max}) = \big\{g \in L^2((a,b); dx) \, \big| \, g \in AC_{loc}((a,b)); \, 
 \beta_{s_b} g \in L^2((a,b); dx)\big\}.   \no 
\end{align}
In addition, the characterizations of $B_{s_b,min}$, $B_{s_b,max}$ analogous to those described in Lemma \ref{l3.1}, 
Theorem \ref{t3.2}, Remark \ref{r3.3}, Lemma \ref{l3.4}, and Lemma \ref{l3.5} hold, of course, and will be used in the following without  repeating them here. 

Finally, we turn to the case of $\tau_{s_a,s_b}$ introduced in \eqref{2.23}, 
\begin{align}
\begin{split} 
\tau_{s_a,s_b} = - \f{d^2}{dx^2} + \f{s_a^2 - (1/4)}{(x-a)^2} + \f{s_b^2 - (1/4)}{(x-b)^2} + q(x), \quad x \in (a,b),& \\
s_a, s_b \in \bbR, \; q \in L^{\infty}((a,b); dx), \; q \text{ real-valued~a.e.~on $(a,b)$.}&     \lb{3.74} 
\end{split}
\end{align}

Recalling that for $\phi \in AC_{loc}((a,b))$,
\begin{equation}
\bigg(- \f{d}{dx} + \phi(x)\bigg)\bigg(\f{d}{dx} + \phi(x)\bigg) = - \f{d^2}{dx^2} + \phi(x)^2 - \phi'(x), \quad x \in (a,b),
\end{equation}
and identifying (for some fixed $\varepsilon \in (0,(b-a)/2)$) 
\begin{equation}
\phi(x) = - \f{s_a + (1/2)}{x-a} \wti \chi_{[a, a+\varepsilon]} + \f{s_b + (1/2)}{b-x} \wti \chi_{[b-\varepsilon,b]}, 
\quad x \in (a,b),     \lb{3.76} 
\end{equation}
one obtains 
\begin{equation}
\phi(x)^2 - \phi'(x) = \f{s_a^2 - (1/4)}{(x-a)^2} + \f{s_b^2 - (1/4)}{(b-x)^2} + \wti q(x), \quad x \in (a,b),   \lb{3.77}
\end{equation}
for some
\begin{equation}
\wti q \in L^{\infty}((a,b); dx), \; \wti q \text{ real-valued a.e.~on $(a,b)$.}     \lb{3.78}
\end{equation}
Relations \eqref{3.76}, \eqref{3.77} define $\wti q$, whose precise form (beyond the properties in \eqref{3.78}) are not important for our purpose. Thus, introducing
\begin{align}
& \alpha_{s_a,s_b} = \f{d}{dx} + \phi(x) 
= \f{d}{dx} - \f{s_a + (1/2)}{x-a} \wti \chi_{[a, a+\varepsilon]} + \f{s_b + (1/2)}{b-x} \wti \chi_{[b-\varepsilon,b]},
\no \\
& \alpha_{s_a,s_b}^+ = - \f{d}{dx} + \phi(x) 
= - \f{d}{dx} - \f{s_a + (1/2)}{x-a} \wti \chi_{[a, a+\varepsilon]} + \f{s_b + (1/2)}{b-x} \wti \chi_{[b-\varepsilon,b]},  \no \\
& \hspace*{7.64cm} s_a, s_b \in \bbR, \; x \in (a,b), 
\end{align}
one concludes that 
\begin{align}
\tau_{s_a,s_b} &= - \f{d^2}{dx^2} + \f{s_a^2 - (1/4)}{(x-a)^2} + \f{s_b^2 - (1/4)}{(x-b)^2} + q(x),   \no \\
&= \alpha_{s_a,s_b}^+  \alpha_{s_a,s_b}^{} + q(x) - \wti q(x), \quad x \in (a,b),      \lb{3.80} \\
& \hspace*{-8mm} 
s_a, s_b \in \bbR, \; q, \wti q \in L^{\infty}((a,b); dx), \; q, \wti q \text{ real-valued~a.e.~on $(a,b)$,}  \no
\end{align}
explaining our interest in the differential expressions $\alpha_{s_a,s_b}^{}$, $ \alpha_{s_a,s_b}^+$. Associated with the latter we thus introduce the operators in $L^2((a,b);dx)$,
\begin{align}
& \dot A_{s_a,s_b,min} f = \alpha_{s_a,s_b} f, \quad s_a, s_b \in \bbR,      \no \\
& f \in\dom\big(\dot A_{s_a,s_b,min}\big) = \big\{g \in L^2((a,b); dx) \, \big| \, g \in AC_{loc}((a,b));   \\ 
&\hspace*{3.85cm} \supp(g) \subset (a,b) \text{ compact}; \,  \alpha_{s_a,s_b} g \in L^2((a,b); dx)\big\},   \no \\
& \ddot A_{s_a,s_b,min} = \alpha_{s_a,s_b}\big|_{C_0^{\infty}((a,b))}, \quad s_a, s_b \in \bbR,   \\
& A_{s_a,s_b,min} = \ol{\dot A_{s_a,s_b,min}} = \ol{\alpha_{s_a,s_b}\big|_{C_0^{\infty}((a,b))}}, 
\quad s_a, s_b \in \bbR,    \\
& A_{s_a,s_b,max} f = \alpha_{s_a,s_b} f, \quad s_a, s_b \in \bbR,   \no \\
& f \in\dom(A_{s_a,s_b,max}) = \big\{g \in L^2((a,b); dx) \, \big| \, g \in AC_{loc}((a,b));   \\ 
& \hspace*{5.8cm} \alpha_{s_a,s_b} g \in L^2((a,b); dx)\big\}.   \no 
\end{align}
 
Employing the local nature of the singularities at $x=a$ and $x=b$ then immediately yields the following 
analog of Theorem \ref{t3.2} (and the analogous theorem for $B_{s_b,min}$, $B_{s_b,max}$) for 
$A_{s_a,s_b,min}$, $A_{s_a,s_b,max}$.

\begin{theorem} \lb{t3.6}
Let $s_a, s_b \in \bbR$. \\[1mm]
$(i)$ For all $s_a, s_b \in \bbR \backslash \{0\}$, 
\begin{equation}
\dom(A_{s_a,s_b,min}) = H_0^1((a,b)).    \lb{3.85}
\end{equation}
$(ii)$ For $s_a = 0$ and $s_b \in \bbR \backslash \{0\}$, or, for $s_a \in \bbR \backslash \{0\}$ and $s_b = 0$, $C_0^{\infty}((a,b))$ is a core for $A_{s_a,s_b,min}$ and $H_0^1((a,b)) \subsetneqq \dom(A_{0,s_b,min})$ and 
$H_0^1((a,b)) \subsetneqq \dom(A_{s_a,0,min})$. In addition,  
\begin{align}
\dom(A_{0,s_b,min}) &= \big\{f \in L^2((a,b); dx) \, \big| \, f \in AC_{loc}((a,b)); \, f(a) = 0 = f(b);   \no \\
& \hspace*{5.42cm} \alpha_{0,s_b} f \in L^2((a,b); dx)\big\}    \lb{3.86} \\
&= \dom(A_{0,s_b,max}), \quad s_b \in \bbR \backslash \{0\},    \no \\
\dom(A_{s_a,0,min}) &= \big\{f \in L^2((a,b); dx) \, \big| \, f \in AC_{loc}((a,b)); \, f(a) = 0 = f(b);    \no \\
& \hspace*{5.45cm}  \alpha_{s_a,0} f \in L^2((a,b); dx)\big\}    \lb{3.87} \\
&= \dom(A_{s_a,0,max}), \quad s_a \in \bbR \backslash \{0\}.   \no 
\end{align} 
In fact, the boundary conditions $f(a)=0=f(b)$ in \eqref{3.86} can be replaced by 
\begin{equation} 
\lim_{x \downarrow a} \f{f(x)}{[(x-a) \ln(R/(x-a))]^{1/2}} = 0 
= \lim_{x \uparrow b} \f{f(x)}{(b-x)^{1/2}},     \lb{3.88}
\end{equation}
and the boundary conditions $f(a)=0=f(b)$ in \eqref{3.87} can be replaced by 
\begin{equation}
\lim_{x \downarrow a} \f{f(x)}{(x-a)^{1/2}}= 0 = 
\lim_{x \uparrow b} \f{f(x)}{[(b-x) \ln(R/(b-x))]^{1/2}} .      \lb{3.89}
\end{equation} 
$(iii)$ For $s_a, s_b \in (-\infty, -1] \cup (0,\infty)$ one has, 
\begin{align}
\dom(A_{s_a,s_b,max}) &= \big\{f \in L^2((a,b); dx) \, \big| \, f \in AC_{loc}((a,b)); \, 
\alpha_{s_a,s_b} f \in L^2((a,b); dx)\big\}    \no \\
& = \big\{f \in L^2((a,b); dx) \, \big| \, f \in AC_{loc}((a,b)); \, f(a) = 0 = f(b);    \lb{3.90} \\
& \hspace*{5.35cm} \alpha_{s_a,s_b} f \in L^2((a,b); dx)\big\}      \no \\
&= \dom(A_{s_a,s_b,min}) = H^1_0((a,b)).    \no
\end{align}
The boundary condition $f(a)=0=f(b)$ in \eqref{3.90} can be replaced by 
\begin{equation}
\lim_{x \downarrow a} \f{f(x)}{(x-a)^{1/2}} = 0 = \lim_{x \uparrow b} \f{f(x)}{(b-x)^{1/2}}.      \lb{3.92}
\end{equation}
In particular,
\begin{equation}
A_{s_a,s_b,max} = A_{s_a,s_b,min}, \quad s_a, s_b \in (-\infty, -1] \cup [0,\infty).  
\end{equation}
$(iv)$ For $s_a \in (-1,0)$ and/or $s_b \in (-1,0)$,  
\begin{equation}
H_0^1((a,b)) = \dom(A_{s_a,s_b,min}) \subsetneqq \dom(A_{s_a,s_b,max}).     \lb{3.94} 
\end{equation}
\end{theorem}

For simplicity we assumed $q \in L^{\infty}((a,b);dx)$ throughout this paper. However, neither the real-valuedness 
of $q$, nor its essential boundedness is needed for the principal domain considerations in this section. All that is required in the end is that $q$ is infinitesimally bounded with respect to the underlying maximal operator in question so that $q$ cannot influence domain considerations.

\section{Domain Properties of Bessel-Type Operators} \lb{s4}

In this section we derive our principal results on (second-order) Bessel-type operators associated with 
$\omega_{s_a} = \alpha_{s_a}^+ \alpha_{s_a}^{}$ (resp., $\eta_{s_b} = \beta_{s_b}^+ \beta_{s_b}^{}$) and 
$\tau_{s_a,s_b} = \alpha_{s_a,s_b}^+  \alpha_{s_a,s_b}^{} + q(\dott) - \wti q(\dott)$. 

Recalling the differential expressions $\omega_{s_a}$ in \eqref{3.2}, 
\begin{equation}
\omega_{s_a} = \alpha_{s_a}^+ \alpha_{s_a}^{} = - \f{d^2}{dx^2} + \f{s_a^2 - (1/4)}{(x-a)^2}, \quad s_a \in [0,\infty), 
\; x \in (a,b),     \lb{4.1} 
\end{equation} 
we turn to the analogous maximal and minimal $L^2$-realizations as well as the Friedrichs extension corresponding to $\omega_{s_a}$, 
$s_a \in [0,\infty)$, employing the boundary values $\wti g(a)$, ${\wti g}^{\, \prime}(a)$ in \eqref{2.12}, \eqref{2.13} at $x=a$ and $g(b)$, $g'(b)$ at $x=b$: 
\begin{align}
&S_{s_a,max} f = \omega_{s_a} f, \quad s_a \in [0,\infty),     \no \\
& f \in \dom(S_{s_a,max})=\big\{g\in L^2((a,b); dx) \, \big| \,  g,g' \in AC_{loc}((a,b));   \lb{4.2} \\
& \hspace*{6.2cm}  \omega_{s_a} g\in L^2((a,b); dx)\big\},     \no \\
&\dot S_{s_a,min} f = \omega_{s_a} f, \quad s_a \in [0,\infty),   \no
\\
& f \in \dom\big(\dot S_{s_a,min}\big)=\big\{g\in L^2((a,b); dx)  \, \big| \,  g,g' \in AC_{loc}((a,b));   
\lb{4.3} \\ 
&\hspace*{1.8cm} \supp \, (g)\subset(a,b) \text{ is compact; } 
 \omega_{s_a} g\in L^2((a,b); dx)\big\},   \no \\
& \ddot S_{s_a,min} = \omega_{s_a}\big|_{C_0^{\infty}((a,b))}, \quad s_a \in [0,\infty),  \lb{4.4} \\
& S_{s_a,min} = \ol{\dot S_{s_a,min}} = \ol{\ddot S_{s_a,min}}, \quad s_a \in [0,\infty),    \lb{4.5} \\ 
&S_{s_a,min} f = \omega_{s_a} f,     \no \\
& f \in \dom(S_{s_a,min})= \big\{g\in L^2((a,b); dx) \, \big| \, g,g' \in AC_{loc}((a,b));    \lb{4.6} \\ 
& \hspace*{1.2cm}  \wti g(a) = {\wti g}^{\, \prime}(a) = g(b) = g^{\prime}(b) = 0; \, 
 \omega_{s_a} g\in L^2((a,b); dx)\big\},  \quad s_a \in [0,1),     \no \\
& S_{s_a,min} = S_{s_a,max}, \quad s_a \in [1,\infty),   \lb{4.7} \\
& S_{s_a,F} f =\omega_{s_a} f,\quad f\in\dom(S_{s_a,F})=\big\{g\in\dom(S_{s_a,max}) \, \big| \, 
\wti g(a) = g(b)=0\big\},     \lb{4.8} \\ 
& \hspace*{9.8cm} s_a \in [0,1),    \no \\ 
& S_{s_a,F} f =\omega_{s_a} f,\quad f\in\dom(S_{s_a,F})=\big\{g\in\dom(S_{s_a,max}) \, \big| \,  g(b)=0\big\}, 
 \lb{4.9}  \\
& \hspace*{8.55cm} s_a \in [1,\infty).   \no
\end{align}

Theorem \ref{t3.2} leads to a somewhat bewilderingly variety of seemingly different, yet obviously equivalent, characterizations of $S_{s_a,F}$ and $\dom\big(S_{s_a,F}^{1/2}\big)$ as follows (we note $S_{s_a,F}\geq\varepsilon I$ for some $\varepsilon>0$ as shown in Theorem \ref{t5.1}):
\begin{align}
S_{s_a,F} &= A_{s_a,min}^* A_{s_a,min}^{} = - A_{-s_a - 1,max} A_{s_a,min}     \lb{4.10} \\
&= S_{s_a,max}\big|_{\dom\big(S_{s_a,F}^{1/2}\big)} = S_{s_a,max}\big|_{\dom(A_{s_a,min})}     \lb{4.11} \\
&=\begin{cases} S_{0,max}\big|_{\dom(A_{0,min})}, & s_a = 0, 
\\[1mm] 
S_{s_a,max}\big|_{H^1_0((a,b))}, & s_a \in (0,\infty), \\[1mm] 
S_{s_a,max}\big|_{\{f \in \dom(S_{s_a,max}) \,|\, \wti f(a) = 0 = f(b)\}}, & s_a \in [0,1), \\
S_{s_a,max}\big|_{\{f \in \dom(S_{s_a,max}) \,|\, f(b) = 0\}}, & s_a \in [1,\infty),
\end{cases}      \lb{4.12} 
\end{align}
and 
\begin{align}
\begin{split} 
& Q_{S_{s_a,F}} (f,g) = (\alpha_{s_a} f, \alpha_{s_a} g)_{L^2((a,b);dx)},     \lb{4.13} \\ 
& f, g \in \dom(Q_{S_{s_a,F}}) = \dom(A_{s_a,min}), \quad s_a \in [0,\infty),   
\end{split} \\
& \dom\big(S_{s_a,F}^{1/2}\big) = \dom(Q_{S_{s_a,F}}) = \dom(A_{s_a,min})   \no \\
& \hspace*{1.8cm} = \begin{cases}
\dom(A_{0,min}), & s_a = 0, \\
H^1_0((a,b)), & s_a \in (0,\infty). 
\end{cases}     \lb{4.14} 
\end{align}
In this context we note that \eqref{4.10} can be found in \cite{BDG11}, \cite{DG21} (for the interval $(0,\infty)$), \cite{GP79}, \cite{Ka78}, \cite{KT11} (also for the interval $(0,\infty)$), for \eqref{4.11} we refer to \cite{EK07}, \cite{GP79}, \cite{Ka78}, for \eqref{4.12} to \cite{GLN20}, 
\cite{Ka78}, \cite{NZ92}, \cite{Ro85}, and in connection with \eqref{4.13} to \cite{AB16}, \cite{BDG11}, \cite{DG21}, \cite{GP79}, \cite{Ka78}. Relations \eqref{4.10}--\eqref{4.14} do by no means exhaust the possible descriptions of $S_{s_a,F}$ and $\dom\big(S_{s_a,F}^{1/2}\big)$ and more can be found, for instance, in \cite{AA12}, \cite{AB15}, \cite{AB16}, \cite{BDG11}, \cite{DF21}, \cite{DG21}, \cite{EK07}, \cite{GP84}, \cite{Ka78}, \cite{KT11}, \cite{Ro85}.  

Interchanging the roles of $x=a$ and $x=b$, that is, interchanging $A_{s_a,min}$, $A_{s_a,max}$  with $B_{s_b,min}$, $B_{s_b,max}$, relations \eqref{4.2}--\eqref{4.14} immediately extend to the singular second order operators generated by 
\begin{equation} 
\eta_{s_b} = \beta_{s_b}^+ \beta_{s_b}^{} = - \f{d^2}{dx^2} + \f{s_b^2 - (1/4)}{(b-x)^2}, 
\quad s_b \in [0,\infty), \; x \in (a,b)
\end{equation} 
(i.e., the analogs of $S_{s_a,min}$, $S_{s_a,max}$, $S_{s_a,F}$, etc., where the inverse square singularity is now at $x=b$ rather than at $x=a$). We refrain from detailing these results.

Instead, employing the local nature of these inverse square singularities at $x=a$ and/or $x=b$, and noticing that 
$q, \wti q \in L^{\infty}((a,b); dx)$ such that $q, \wti q$ can be ignored in connection with domain questions, one thus obtains the following analogs of \eqref{4.2}--\eqref{4.14} when $\omega_{s_a}$, respectively, $\eta_{s_b}$, is replaced by $\tau_{s_a,s_b}$, where 
\begin{align}
\tau_{s_a,s_b} &= - \f{d^2}{dx^2} + \f{s_a^2 - (1/4)}{(x-a)^2} + \f{s_b^2 - (1/4)}{(x-b)^2} + q(x),   \no \\
&= \alpha_{s_a,s_b}^+  \alpha_{s_a,s_b}^{} + q(x) - \wti q(x), \quad x \in (a,b),      \lb{4.16} \\
& \hspace*{-8mm} 
s_a, s_b \in \bbR, \; q, \wti q \in L^{\infty}((a,b); dx), \; q, \wti q \text{ real-valued~a.e.~on $(a,b)$,}  \no
\end{align}
with $\wti q$ introduced via \eqref{3.76}, \eqref{3.77}.

\begin{theorem} \lb{t4.1} 
Let $s_a, s_b \in [0,\infty)$, then, 
\begin{align}
&T_{s_a,s_b,max} f = \tau_{s_a,s_b} f, \quad s_a, s_b \in [0,\infty),     \no \\
& f \in \dom(T_{s_a,s_b,max})=\big\{g\in L^2((a,b); dx) \, \big| \,  g,g' \in AC_{loc}((a,b));   \lb{4.17} \\
& \hspace*{6.25cm}  \tau_{s_a,s_b} g\in L^2((a,b); dx)\big\},     \no \\
&\dot T_{s_a,s_b,min} f = \tau_{s_a,s_b} f, \quad s_a, s_b \in [0,\infty),   \no
\\
& f \in \dom\big(\dot T_{s_a,s_b,min}\big)=\big\{g\in L^2((a,b); dx)  \, \big| \,  g,g' \in AC_{loc}((a,b));   
\lb{4.18} \\ 
&\hspace*{1.85cm} \supp \, (g)\subset(a,b) \text{ is compact; } 
\tau_{s_a,s_b} g\in L^2((a,b); dx)\big\},   \no \\
& \ddot T_{s_a,s_b,min} = \tau_{s_a,s_b}\big|_{C_0^{\infty}((a,b))}, \quad s_a, s_b \in [0,\infty),  \lb{4.19} \\
& T_{s_a,s_b,min} = \ol{\dot T_{s_a,s_b,min}} = \ol{\ddot T_{s_a,s_b,min}}, \quad s_a, s_b \in [0,\infty),    \lb{4.20} \\ 
&T_{s_a,s_b,min} f = \tau_{s_a,s_b} f,     \no \\
& f \in \dom(T_{s_a,s_b,min})   \lb{4.21} \\ 
& \hspace*{6mm} = \left\{g\in \dom(T_{s_a,s_b,max}) \, \middle \vert \; \begin{matrix} 
\wti g(a) = {\wti g}^{\, \prime}(a) = \wti g(b) = {\wti g}^{\, \prime}(b) = 0, \;\, s_a, s_b \in [0,1),  \\
\hspace*{-8.5mm} \wti g(a) = {\wti g}^{\, \prime}(a) = 0, \;\, s_a \in [0,1), \; s_b \in [1,\infty), \\
\hspace*{-9.3mm} \wti g(b) = {\wti g}^{\, \prime}(b) = 0, \;\, s_a \in [1,\infty), \; s_b \in [0,1) \end{matrix} 
\right\},  \no \\ 
& T_{s_a,s_b,min} = T_{s_a,s_b,max}, \quad s_a, s_b \in [1,\infty),   \lb{4.22} \\
& T_{s_a,s_b,F} f =\tau_{s_a,s_b} f,     \no \\
& f\in\dom(T_{s_a,s_b,F})      \lb{4.23}  \\
&  \hspace*{6mm} = \left\{g\in\dom(T_{s_a,s_b,max}) \, \middle| \; \begin{matrix} 
\hspace*{-3.5mm}  \wti g(a) = \wti g(b) = 0, \;\, s_a, s_b \in [0,1),  \\ 
 \wti g(a) = 0, \;\, s_a \in [0,1), \; s_b \in [1,\infty),  \\ 
 \wti g(b)=0, \;\, s_a \in [1,\infty), \; s_b \in [0,1). \end{matrix} \right\},    \no \\
& T_{s_a,s_b,F} = T_{s_a,s_b,max}, \quad s_a, s_b \in [1,\infty).      \lb{4.24} 
\end{align}
\end{theorem}

Moreover, the analog of Theorem \ref{t3.2} then leads to the following characterizations of $T_{s_a,s_b,F}$ and its form domain, $\dom\big(|T_{s_a,s_b,F}|^{1/2}\big)$.

\begin{theorem} \lb{t4.2}  
Let $s_a, s_b \in [0,\infty)$, then, 
\begin{align}
T_{s_a,s_b,F} &= A_{s_a,s_b}^* A_{s_a,s_b}^{} + q - \wti q = A_{-s_a-1,-s_b-1} A_{s_a,s_b}^{} + q - \wti q      \\ 
&= T_{s_a,s_b,max}\big|_{\dom(|T_{s_a,s_b,F}|^{1/2})} 
= T_{s_a,s_b,max}\big|_{\dom(A_{s_a,s_b,min})}      \\
&= \begin{cases} T_{0,0,max}\big|_{\dom(A_{0,0,min})}, & s_a = s_b = 0, \\[1mm] 
T_{0,s_b,max}\big|_{\dom(A_{0,s_b,min})}, & s_a = 0, \,  s_b \in (0,\infty), \\[1mm] 
T_{s_a,0,max}\big|_{\dom(A_{s_a,0,min})}, & s_a \in (0,\infty), \,  s_b = 0, \\[1mm] 
T_{s_a,s_b,max}\big|_{H^1_0((a,b))}, & s_a, s_b \in (0,\infty), \\[1mm]
T_{s_a,s_b,max}, &  s_a, s_b \in [1,\infty), 
\end{cases}      \lb{4.27} 
\end{align}
and 
\begin{align}
\begin{split} 
& Q_{T_{s_a,s_b,F}} (f,g) = (\alpha_{s_a,s_b} f, \alpha_{s_a,s_b} g)_{L^2((a,b);dx)} 
+ \big(f, \big[q - \wti q\big]g\big)_{L^2((a,b);dx)},     \lb{4.28} \\ 
& f, g \in \dom(Q_{T_{s_a,s_b,F}}) = \dom(A_{s_a,s_b,min}), \quad s_a, s_b \in [0,\infty),   
\end{split} \\
& \dom\big(|T_{s_a,s_b,F}|^{1/2}\big) = \dom(Q_{T_{s_a,s_b,F}}) = \dom(A_{s_a,s_b,min})    \no \\
& \hspace*{1.8cm} = \begin{cases}
\dom(A_{0,s_b,min}), &s_a=0, \, s_b \in (0,\infty),  \\
\dom(A_{s_a,0,min}), &s_a \in (0,\infty), \, s_b = 0,  \\
H^1_0((a,b)), & s_a, s_b \in (0,\infty). 
\end{cases}     \lb{4.29} 
\end{align}
\end{theorem}

\begin{remark} \lb{r4.3}
While $\wti f(a) = 0$, respectively, $\wti f(b) = 0$, represent the boundary conditions in the domain of 
$S_{s_a,F}$, $T_{s_a,s_b,F}$, etc., we want to point out that their primary role is to assure that functions $f$ in 
$\dom(S_{s_a,F})$, $\dom(T_{s_a,s_b,F})$, etc., in a neighborhood of $x=a$, respectively, $x=b$, look like the corresponding principal solution at the endpoint in question. In other words, in a neighborhood of $x=a,b$, these boundary conditions filter out the principal solution behavior and discard the behavior of nonprincipal solutions.  However, in performing this filtering procedure, the boundary conditions mask the actual behavior of elements in $\dom(S_{s_a,F})$, $\dom(T_{s_a,s_b,F})$, in the following sense: Focusing for brevity on the left endpoint $a$, the condition 
\begin{equation}
\wti f (a) = 0 \, \text{ is equivalent to } \, \begin{cases} 
\lim_{x \downarrow a} \f{f(x)}{(x-a)^{1/2} \ln(1/(x-a))} = 0, & s_a = 0, \\
\lim_{x \downarrow a} \f{f(x)}{(x-a)^{(1/2) - s_a}} = 0, & s_a \in (0,1),
\end{cases}       \lb{4.30} 
\end{equation}
and indeed discards nonprincipal solutions at $a$ which behave like (cf.\ \eqref{2.25})
\begin{equation}
\begin{cases}
(x-a)^{1/2} \ln(1/(x-a)), & s_a = 0, \\
(x-a)^{(1/2) - s_a}, & s_a \in (0,1),
\end{cases}
\end{equation}
as opposed to principal solutions at $a$ which are of the form (cf.\ \eqref{2.24})
\begin{equation}
(x-a)^{(1/2) + s_a}, \quad s_a \in [0,1),
\end{equation}
for $x$ sufficiently close to $a$. However, for $s_a \in (1/2,1)$, \eqref{4.30} is actually of the form
\begin{equation}
\lim_{x \downarrow a} (x-a)^{s_a - (1/2)}f(x) = 0, \quad s_a \in (1/2,1), 
\end{equation}
and hence does not necessarily enforce $f(0) = 0$, even though we know from \eqref{B.6} that actually 
much more holds as 
\begin{equation}
\lim_{x \downarrow a} \f{f(x)}{(x-a)^{1/2}} = 0, \quad s_a \in (0,\infty). 
\end{equation}
${}$ \hfill $\diamond$ 
\end{remark}

\section{The Krein--von Neumann Extension of $S_{s_a,min}$ and $T_{s_a,s_b,min}$} \lb{s5}

The Friedrichs and Krein--von Neumann extensions are of special importance as extremal nonnegative extensions of $S_{s_a,min}$ and $T_{s_a,s_b,min}$, respectively (assuming $T_{s_a,s_b,min} \geq 0$). In the present section we explicitly describe their Krein-von Neumann extensions utilizing Theorem \ref{t2.8} (noting that the explicit form of the corresponding Friedrichs extensions was detailed in Section \ref{s4}). 

\begin{theorem} \lb{t5.1} 
Let $s_a \in[0,\infty)$. Then $S_{s_a,min} \geq \varepsilon I$ for some $\varepsilon > 0$ and 
the Krein--von Neumann extension of $S_{s_a,min}$ is given by
\begin{align}
& S_{s_a,\b_K}f = \omega_{s_a} f,   \no \\
& f \in \dom(S_{s_a,\b_K})=\big\{g\in\dom(S_{s_a,max}) \, \big| \, \sin(\b_K)  g'(b) 
+ \cos(\b_K) g(b) = 0\big\},   \lb{4.35}  \\
& \cot(\b_K) = - (s_a+(1/2)) / (b-a), \quad \b_K \in (0,\pi),\;  s_a \in[1,\infty),   \no \\
\begin{split} 
& S_{s_a,0,R_K} f = \omega_{s_a} f,    \\
& f \in \dom(S_{s_a,0,R_K})=\bigg\{g\in\dom(S_{s_a,max}) \, \bigg| \begin{pmatrix} g(b) 
\\ g'(b) \end{pmatrix} = R_K \begin{pmatrix}
\wti g(a) \\ {\wti g}^{\, \prime}(a) \end{pmatrix} \bigg\}, \\
& \hspace*{8.95cm} s_a\in[0,1),
\end{split}    \lb{4.36}
\end{align}
where 
\begin{align}
R_{K}&=\begin{cases}
\begin{pmatrix}  \dfrac{(b-a)^{(1/2)-s_a}}{2s_a} & (b-a)^{(1/2)+s_a} \\
\left(\dfrac{1}{4s_a}-\dfrac{1}{2}\right)(b-a)^{-(1/2)-s_a} & \left(\dfrac{1}{2}+s_a\right) (b-a)^{-(1/2)+s_a}
\end{pmatrix},\\
\hfill s_a\in(0,1), \\[1mm]
\begin{pmatrix}  (b-a)^{1/2}\ln(1/(b-a)) & (b-a)^{1/2} \\
\dfrac{\ln(1/(b-a))-2}{2(b-a)^{1/2}} & \dfrac{1}{2(b-a)^{1/2}}
\end{pmatrix},\quad s_a=0.
\end{cases}    \lb{4.37}
\end{align}\end{theorem}
\begin{proof}
Let $s_a \in[0,\infty)$. We first prove that $S_{s_a,min} \geq \varepsilon I$ for some $\varepsilon > 0$. Since 
\begin{equation}
\omega_{s_a}|_{C^{\infty}_0((a,b))} \geq \omega_{0}|_{C^{\infty}_0((a,b))}, 
\end{equation}
it suffices to prove that $S_{0,min} \geq \varepsilon I$ for some $\varepsilon > 0$. Noting that the smallest eigenvalue of the Friedrichs extension, $S_{0,F}$, of $S_{0,min}$ is $j_{0,1}^2/(b-a)^2$, where $j_{\nu,k}$ is the $k$th positive zero of the Bessel function $J_\nu(\dott)$ of order $\nu\in\bbR$ $($see, e.g., \cite[Sect.~9.5]{AS72}, \cite[Ch.~XV]{Wa96}$)$ proves $S_{s_a,min} \geq \varepsilon I$ for $\varepsilon\in\big(0,j_{0,1}^2/(b-a)^2\big)$.

Let $u_{a,s_a}(0,\dott),\hatt u_{a,s_a}(0,\dott)$ be defined as in \eqref{2.24}, \eqref{2.25} respectively. Then \eqref{4.35} is an immediate result of \eqref{2.19} in Theorem \ref{t2.8} noting that by \eqref{2.24},
\begin{equation}
- \dfrac{u_{a,s_a}'(0,b)}{u_{a,s_a}(0,b)}=-\dfrac{s_a+(1/2)}{b-a},\quad s_a\in[1,\infty).
\end{equation}

Furthermore, \eqref{4.37} is a special case of the generalized Bessel operator treated in \cite[Example 4.1]{FGKLNS21} (with $\a=\b=0,\;  \g=s_a$ and $x\mapsto (x-a)$).
In particular, \eqref{2.22} yields \eqref{4.36}--\eqref{4.37}. One verifies that $\det(R_K) =1$.
\end{proof}

\begin{theorem}\lb{t5.2}
Let $s_a, s_b \in[0,\infty)$ and suppose that for some $\varepsilon > 0$, $T_{s_a,s_b,min} \geq \varepsilon I$. Denote the principal and nonprincipal solutions of $\tau_{s_a,s_b}u=0$ at $a$ and $b$ by $u_{a,s_a,s_b}(0,\dott),\hatt u_{a,s_a,s_b}(0,\dott)$ and $u_{b,s_a,s_b}(0,\dott), \hatt u_{b,s_a,s_b}(0,\dott)$, respectively. Then the Krein--von Neumann extension of $T_{s_a,s_b,min}$ is given by
\begin{align}
& T_{s_a,s_b,\a_K}f = \tau_{s_a,s_b} f,  \lb{4.40} \\
& f \in \dom(T_{s_a,s_b,\a_K})=\big\{g\in\dom(T_{s_a,s_b,max}) \, \big| \, \sin(\a_K) {\wti g}^{\, \prime}(b) 
+ \cos(\a_K) \wti g(b) = 0\big\},   \no  \\
& \cot(\a_K) = - \wti u_{b,s_a,s_b}^{\, \prime}(0,a) / \wti u_{b,s_a,s_b}(0,a), \quad \a_K \in (0,\pi),\;  s_a \in[0,1),\;  s_b\in[1,\infty),   \no \\
& T_{s_a,s_b,\b_K}f = \tau_{s_a,s_b} f,   \lb{4.41}   \\
& f \in \dom(T_{s_a,s_b,\b_K})=\big\{g\in\dom(T_{s_a,s_b,max}) \, \big| \, \sin(\b_K) {\wti g}^{\, \prime}(b) 
+ \cos(\b_K) \wti g(b) = 0\big\},   \no  \\
& \cot(\b_K) = - \wti u_{a,s_a,s_b}^{\, \prime}(0,b) / \wti u_{a,s_a,s_b}(0,b), \quad \b_K \in (0,\pi),\;  s_a \in[1,\infty),\;  s_b\in[0,1),   \no \\
& T_{s_a,s_b,K} = T_{s_a,s_b,max}, \quad s_a,s_b \in [1,\infty),\\
\begin{split} 
& T_{s_a,s_b,0,R_K} f = \tau_{s_a,s_b} f,    \\
& f \in \dom(T_{s_a,s_b,0,R_K})=\bigg\{g\in\dom(T_{s_a,s_b,max}) \, \bigg| \begin{pmatrix} \wti g(b) 
\\ {\wti g}^{\, \prime}(b) \end{pmatrix} = R_K \begin{pmatrix}
\wti g(a) \\ {\wti g}^{\, \prime}(a) \end{pmatrix} \bigg\}, \\
& \hspace*{9.15cm} s_a, s_b\in[0,1),
\end{split}  \lb{4.43}
\end{align}
where 
\begin{align}
R_K=\begin{pmatrix} \wti{\hatt u}_{a,s_a,s_b}(0,b) & \wti u_{a,s_a,s_b}(0,b)\\
\wti{\hatt u}_{a,s_a,s_b}^{\, \prime}(0,b) & \wti u_{a,s_a,s_b}^{\, \prime}(0,b)
\end{pmatrix}.  \lb{4.44}
\end{align}
\end{theorem}
\begin{proof}
A direct consequence of Theorem \ref{t2.8}. \end{proof}

For a concrete illustration of Theorem \ref{t5.2} in the special case $q = 0$ we refer to Appendix \ref{sA}.

\appendix 

\section{The example $q=0$} \lb{sA} 

In this appendix we illustrate Theorem \ref{t5.2} in the special case $q = 0$. 

We start by relating $\tau_{s_a,s_b,q=0}u = z u$ to the confluent Heun differential equation (see, e.g.,  
\cite[no.~31.12.1]{OLBC10}, \cite[Part~B]{Ro95}, \cite[Ch.~3]{SL00}),
\begin{align}
& w''(\xi)+\left(\dfrac{\g}{\xi}+\dfrac{\d}{\xi-1}+\ve\right)w'(\xi)+\dfrac{\nu \xi- \mu}{\xi(\xi-1)}w(\xi)=0,\quad \g,\d,\ve,\mu, \nu,\in\bbC,\; \xi\in(0,1).       \lb{A.0}
\end{align}
One observes that the confluent Heun differential equation has regular singularities at $\xi=0,1$ and an irregular singularity of rank 1 at $\xi=\infty$. Eliminating the first-order term $w'$ in the standard manner by introducing the change of dependent variable
\begin{align}
w \mapsto v, \quad
v(\xi)=e^{\varepsilon \xi/2}\xi^{\g/2}(\xi-1)^{\d/2}w(\xi), \; \xi \in (0,1),      \lb{A.0a}
\end{align}
transforms \eqref{A.0} into the normal form
\begin{align}
v''(\xi)+\left(A+\dfrac{B}{\xi}+\dfrac{C}{\xi-1}+\dfrac{D}{\xi^2}+\dfrac{E}{(\xi-1)^2}\right)v(\xi)=0, 
\lb{A.0b}
\end{align}
where
\begin{align}
\begin{split}
& A=- \ve^2/4 ,\quad B= [2\mu+(\d-\ve)]/2,\quad C= [2\nu-(\ve+\g)\d-2\mu]/2,\\
& D= (2-\g)\g/4,\quad E= (2-\d)\d/4.   \lb{A.0c}
\end{split}
\end{align}
To facilitate the comparison of \eqref{A.0b} with $\tau_{s_a,s_b,q=0} u = z u$, it suffices to set 
\begin{align}
&\g=1+ 2s_a,\quad \d=1- 2s_b, \quad \ve=2i(a-b)z^{1/2},     \lb{A.0d} \\
& \mu=(1/2)(1+2s_a)\big[2i(a-b)z^{1/2}+2s_b-1\big], \quad \nu=2i(a-b)z^{1/2}(1+s_a-s_b),    \no 
\end{align} 
such that the choice 
\begin{equation}
A=z, \quad B=C=0, \quad D=(1/4)-s_a^2, \quad E=(1/4)-s_b^2,     \lb{A.0e} 
\end{equation}
in \eqref{A.0b} combined with the variable changes
\begin{equation}
(0,1) \ni \xi\mapsto(x-a)/(b-a), \quad v(\xi) = u(x), \quad x \in (a,b),     \lb{A.2}
\end{equation} 
yields equivalence of \eqref{A.0b}, \eqref{A.0e}, \eqref{A.2} and $\tau_{s_a,s_b,q=0} u = z u$.

This shows that the ``two-point'' Bessel-type differential equation $\tau_{s_a,s_b,q=0} u = z u$ is a special case of the confluent Heun differential equation. These considerations extend to a constant potential term $q(x)=q_0\in\bbR$, we omit the details. 

Returning to \eqref{A.0} in the special case $\varepsilon = 0$, that is, $A=z=0$, the resulting choice 
\begin{align}
&\g=1+ 2s_a,\quad \d=1- 2s_b, \quad \ve=0,     \lb{A.2a} \\
& \mu=(1/2)(1+2s_a)(2s_b-1), \quad \nu=0,    \no 
\end{align} 
finally reduces \eqref{A.0} to the hypergeometric differential equation $($see \cite[no.~15.5.1]{AS72}$)$
\begin{align}
w''(\xi)+\left(\dfrac{1+2s_a}{\xi}+\dfrac{1-2s_b}{\xi-1}\right)w'(\xi)+\dfrac{(1+2s_a)(1-2s_b)}{2\xi(\xi-1)}w(\xi)=0,\quad \xi\in(0,1).     \lb{A.2b}
\end{align}
Thus, the solutions of $\tau_{s_a,s_b,q = 0}u=0$ can be expressed in terms of appropriate hypergeometric functions (see \eqref{4.49}--\eqref{4.52}). 

\begin{example} \lb{eA.1} 
Suppose $s_a, s_b \in [0,\infty)$ and $q = 0$. Then $T_{s_a,s_b,q=0,min} \geq c_0 I$ for some $c_0 > 0$. Indeed, since 
\begin{equation}
\tau_{s_a,s_b,q=0}|_{C^{\infty}_0((a,b))} \geq \tau_{0,0,q=0}|_{C^{\infty}_0((a,b))}, 
\end{equation}
it suffices to prove that $T_{0,0,q=0,min} \geq c_0 I$ for some $c_0 > 0$. Without loss of generality, we temporarily consider the special case $a = 0$, $b = \pi$ and recall that 
\begin{equation}
\bigg(- \f{d^2}{dx^2} - \f{1}{4 \sin^2(x)} - \f{1}{4}\bigg)\bigg|_{C_0^{\infty}((0,\pi))} \geq 0,
\end{equation}
$($see, e.g., \cite{GK85}, \cite{GMS85}, \cite{GPS21}, \cite{Sc58} and the literature referenced therein\,$)$.~The claim $T_{0,0,q=0,min} \geq c_0 I$ for some $c_0 > 0$ now follows from the fact that for appropriate $c_0 > 0$, 
\begin{equation}
x^{-2} + (\pi - x)^{-2} + c_0 < [\sin(x)]^{-2} + 1, \quad x \in (0,\pi).   \lb{A.3}
\end{equation}
To verify \eqref{A.3} one introduces
\begin{equation}
f(x) = [\sin(x)]^{-2} + 1-x^{-2} - (\pi - x)^{-2},\quad x\in(0,\pi),
\end{equation}
and notes that the minimum of $f$ on the interval $(0,\pi)$ occurs precisely at $x=\pi/2$. Hence, \eqref{A.3} holds 
for $c_0 \in\big(0,2 - \big(8/\pi^2\big)\big)$, proving $T_{s_a,s_b,q=0,min} \geq c_0 I$, with $c_0 > 0$.

Utilizing \eqref{A.0a} and \eqref{A.2a}, \eqref{A.2b}, one finds the following explicit expressions for the principal and nonprincipal solutions of $\tau_{s_a,s_b,q=0} u = 0$:  
\begin{align}
&u_{a,s_a,s_b}(0, x; q = 0)= (b-a)^{s_b-(1/2)}(x-a)^{(1/2) + s_a}(b-x)^{(1/2)-s_b}   \no\\
& \times F((1/2)+s_a-s_b+\sigma_{s_a,s_b},(1/2)+s_a-s_b-\sigma_{s_a,s_b};1+2s_a;(x-a)/(b-a)),   \no \\
& \hspace*{9.5cm} s_a \in [0,1), \lb{4.49}   \\
&\hatt u_{a,s_a,s_b}(0, x; q = 0) = \begin{cases} (2 s_a)^{-1}(b-a)^{s_b-(1/2)} (x-a)^{(1/2) - s_a}(b-x)^{(1/2)-s_b}    \\
\times F((1/2)-s_a-s_b+\sigma_{s_a,s_b},(1/2)-s_a-s_b-\sigma_{s_a,s_b};  \\
\qquad\;\; 1-2s_a;(x-a)/(b-a)), \quad s_a \in (0,1), \\
(b-a)^{s_b-(1/2)}(x-a)^{1/2}\ln(1/(x-a))(b-x)^{(1/2)-s_b}   \\
\times F((1/2)-s_b+\sigma_{0,s_b},(1/2)-s_b-\sigma_{0,s_b};1;\\
\qquad\;\; (x-a)/(b-a)), \quad s_a =0, \end{cases} \no \\
& \hspace*{9.5cm} s_b \in [0,1),   \lb{4.50} \\
&u_{b,s_a,s_b}(0, x; q = 0)= (b-a)^{s_a-(1/2)}(x-a)^{(1/2) -s_a}(b-x)^{(1/2)+s_b}   \no\\
& \times F((1/2)-s_a+s_b+\sigma_{s_a,s_b},(1/2)-s_a+s_b-\sigma_{s_a,s_b};1+2s_b;(b-x)/(b-a)),   \no \\
& \hspace*{9.5cm} s_b \in [0,1), \lb{4.51}  \\
&\hatt u_{b,s_a,s_b}(0, x; q = 0) = \begin{cases} -(2 s_b)^{-1}(b-a)^{s_a-(1/2)} (x-a)^{(1/2) - s_a}(b-x)^{(1/2)-s_b}    \\
\times F((1/2)-s_a-s_b+\sigma_{s_a,s_b},(1/2)-s_a-s_b-\sigma_{s_a,s_b};  \\
\qquad\;\; 1-2s_b;(b-x)/(b-a)), \quad s_b \in (0,1), \\
(b-a)^{s_a-(1/2)}(x-a)^{(1/2)-s_a}(b-x)^{1/2} \ln(1/(b-x))   \\
\times F((1/2)-s_a+\sigma_{s_a,0},(1/2)-s_a-\sigma_{s_a,0};1;\\
\qquad\;\; (b-x)/(b-a)), \quad s_b =0, \end{cases} \no \\
& \hspace*{9.5cm} s_a \in [0,1),  \lb{4.52}
\end{align}
where
\begin{align}
\sigma_{s_a,s_b}=(1/2)\big(4 s_a^2+4 s_b^2-1\big)^{1/2}.
\end{align}
Here $F(\dott,\dott;\dott;\dott)$ denotes the hypergeometric function $($see, e.g., \cite[Ch.~15]{AS72}$)$.
Recalling Gauss's identity $($cf.\ \cite[no.~15.1.20]{AS72}$)$
\begin{equation}
F(\alpha,\beta;\gamma;1) = \f{\Gamma(\gamma) \Gamma(\gamma - \alpha - \beta)}{\Gamma(\gamma - \alpha) \Gamma(\gamma - \beta)}, \quad \gamma \in\bbC \backslash \{ - \bbN_0\}, 
\; \Re(\gamma - \alpha - \beta) > 0,     \lb{A.10} 
\end{equation}
$($here $\gamma$ is not to be confused with $\gamma$ in \eqref{A.0}$)$  
it is easily verified that the two $F(\dott,\dott;\dott;1)$ in \eqref{4.49} and \eqref{4.50} obtained as $x$ approaches $b$ exist for $s_b>0$. Similarly, the two $F(\dott,\dott;\dott;1)$ in \eqref{4.51} and \eqref{4.52} obtained as $x$ approaches $a$ exist for $s_a>0$. Furthermore, $\sigma_{s_a,s_b}$ is strictly imaginary whenever $s_a^2+s_b^2<1/4$, real-valued for $s_a^2+s_b^2 \geq 1/4$, and for the choices of $\alpha, \beta, \gamma$ in \eqref{A.10} with values of 
$\alpha, \beta, \gamma$ taken from \eqref{4.49}--\eqref{4.52} one has $\gamma, (\gamma - \alpha - \beta) \in \bbR$ and $\ol{\gamma - \alpha} = \gamma - \beta$ whenever $s_a^2+s_b^2<1/4$. Thus, noting that 
$\Gamma(\ol{z})=\ol{\Gamma(z)}$, one confirms that each $F(\dott,\dott;\dott;1)$ in \eqref{4.49}--\eqref{4.52} is real-valued, as expected.

Furthermore, one notes that $u_{a,s_a,s_b}(0,\dott;q = 0)$ and $u_{b,s_a,s_b}(0,\dott;q = 0)$ are in $L^2((a,b); dx)$ for $s_a\in[0,\infty)$ and $s_b\in[0,\infty)$, respectively, hence \eqref{4.49} and \eqref{4.51} yield the principal solutions
\begin{align}
&u_{a,s_a,s_b}(0, x; q = 0)= (b-a)^{s_b-(1/2)}(x-a)^{(1/2) + s_a}(b-x)^{(1/2)-s_b}   \no\\
& \times F((1/2)+s_a-s_b+\sigma_{s_a,s_b},(1/2)+s_a-s_b-\sigma_{s_a,s_b};1+2s_a;(x-a)/(b-a)),  \no \\
& \hspace*{7.3cm} s_a \in [0,\infty),\;  s_b \in [0,1),   \lb{4.55}\\
&u_{b,s_a,s_b}(0, x; q = 0)= (b-a)^{s_a-(1/2)}(x-a)^{(1/2) -s_a}(b-x)^{(1/2)+s_b}   \no\\
& \times F((1/2)-s_a+s_b+\sigma_{s_a,s_b},(1/2)-s_a+s_b-\sigma_{s_a,s_b};1+2s_b;(b-x)/(b-a)),   \no \\
& \hspace*{7.3cm} s_a \in [0,1),\;  s_b \in [0,\infty).   \lb{4.56}
\end{align}

We begin by focusing on the leading behavior of these principal solutions in order to compute the analog of \eqref{4.40} and \eqref{4.41} for this example. Employing the connection formula found in \cite[no.~15.3.6]{AS72} yields the leading behavior
\begin{align}
&u_{a,s_a,s_b}(0, x; q=0) \underset{x \uparrow b}{=}\dfrac{(b-a)^{s_a+ s_b}(b-x)^{(1/2) - s_b}\Gamma(1+2s_a)\Gamma(2s_b)}{\Gamma((1/2)+s_a+s_b+\sigma_{s_a,s_b})\Gamma((1/2)+s_a+s_b-\sigma_{s_a,s_b})}   \no \\
& \hspace*{3.4cm} \quad \times (1+\Oh(b-x))   \no \\
& \hspace*{3.4cm} +\dfrac{(b-a)^{ s_a-s_b}(b-x)^{(1/2) + s_b}\Gamma(1+2s_a)\Gamma(-2s_b)}{\Gamma((1/2)+s_a-s_b+\sigma_{s_a,s_b})\Gamma((1/2)+s_a-s_b-\sigma_{s_a,s_b})}    \no \\
& \hspace*{3.4cm} \quad \times (1+\Oh(b-x)), \quad s_a \in [0,\infty),\;  s_b \in (0,1),   \\
&u_{b,s_a,s_b}(0, x; q=0)\underset{x \downarrow a}{=}\dfrac{(b-a)^{s_a+ s_b}(x-a)^{(1/2) - s_a}\Gamma(1+2s_b)\Gamma(2s_a)}{\Gamma((1/2)+s_a+s_b+\sigma_{s_a,s_b})\Gamma((1/2)+s_a+s_b-\sigma_{s_a,s_b})}   \no \\
& \hspace*{3.4cm} \quad \times (1+\Oh(x-a))   \no \\
& \hspace*{3.4cm} +\dfrac{(b-a)^{s_b-s_a}(x-a)^{(1/2) + s_a}\Gamma(1+2s_b)\Gamma(-2s_a)}{\Gamma((1/2)-s_a+s_b+\sigma_{s_a,s_b})\Gamma((1/2)-s_a+s_b-\sigma_{s_a,s_b})}    \no \\
& \hspace*{3.4cm} \quad \times (1+\Oh(x-a)), \quad s_a \in (0,1),\;  s_b \in [0,\infty).
\end{align}
For $s_b=0$ in \eqref{4.55} and $s_a=0$ in \eqref{4.56}, one instead uses \cite[no.~15.3.10]{AS72} to find the leading behavior
\begin{align}
&u_{a,s_a,0}(0, x; q=0) \underset{x \uparrow b}{=}\dfrac{(b-a)^{s_a}(b-x)^{1/2}\Gamma(1+2s_a)}{\Gamma((1/2)+s_a+\sigma_{s_a,0})\Gamma((1/2)+s_a-\sigma_{s_a,0})} [\ln(b-a)-2\gamma_{E} \no \\
&\hspace*{3.3cm}-\psi((1/2)+s_a+\sigma_{s_a,0})-\psi((1/2)+s_a-\sigma_{s_a,0})-\ln(b-x)]  \no \\
& \hspace*{3.3cm} \quad \times (1+\Oh(b-x)), \quad s_a \in [0,\infty),\\
&u_{b,0,s_b}(0, x; q=0)\underset{x \downarrow a}{=}\dfrac{(b-a)^{s_b}(x-a)^{1/2}\Gamma(1+2s_b)}{\Gamma((1/2)+s_b+\sigma_{0,s_b})\Gamma((1/2)+s_b-\sigma_{0,s_b})} [\ln(b-a)-2\gamma_{E} \no \\
&\hspace*{3.3cm}-\psi((1/2)+s_b+\sigma_{0,s_b})-\psi((1/2)+s_b-\sigma_{0,s_b})-\ln(x-a)]  \no \\
& \hspace*{3.3cm} \quad \times (1+\Oh(x-a)), \quad s_b \in [0,\infty).
\end{align}
Here $\psi(\dott) = \Gamma'(\dott)/\Gamma(\dott)$ is the Digamma function and $\gamma_{E} = - \psi(1) = 0.57721\dots$ represents Euler's constant $($see, e.g., \cite[Ch.~6]{AS72}$)$. Hence using \eqref{2.28}--\eqref{2.31} one readily computes
\begin{align}
&\wti u_{a,s_a,s_b}(0, b; q=0)= \dfrac{(b-a)^{s_a+ s_b}\Gamma(1+2s_a)\Gamma(1+2s_b)}{\Gamma((1/2)+s_a+s_b+\sigma_{s_a,s_b})\Gamma((1/2)+s_a+s_b-\sigma_{s_a,s_b})},   \no \\
& \hspace*{7.3cm} s_a \in [0,\infty), \; s_b \in [0,1),   \lb{4.61}  \\
&\wti u^{\, \prime}_{a,s_a,s_b}(0, b; q=0) =
\begin{cases}\dfrac{-(b-a)^{s_a-s_b}\Gamma(1+2s_a)\Gamma(-2s_b)}{\Gamma((1/2)+s_a-s_b+\sigma_{s_a,s_b})\Gamma((1/2)+s_a-s_b-\sigma_{s_a,s_b})},  \\
\hspace*{4.8cm} s_a \in [0,\infty), \; s_b\in(0,1),   \\
\dfrac{-(b-a)^{s_a}\Gamma(1+2s_a)}{\Gamma((1/2)+s_a+\sigma_{s_a,0})\Gamma((1/2)+s_a-\sigma_{s_a,0})}[\ln(b-a)   
\\
-2\gamma_{E}-\psi((1/2)+s_a+\sigma_{s_a,0})-\psi((1/2)+s_a-\sigma_{s_a,0})],\\
\hfill s_a \in [0,\infty), \; s_b=0,
\end{cases}  \lb{4.62}   \\
&\wti u_{b,s_a,s_b}(0, a; q=0)= \dfrac{(b-a)^{s_a+ s_b}\Gamma(1+2s_a)\Gamma(1+2s_b)}{\Gamma((1/2)+s_a+s_b+\sigma_{s_a,s_b})\Gamma((1/2)+s_a+s_b-\sigma_{s_a,s_b})} ,   \no \\
& \hspace*{7.3cm} s_a \in [0,1),\; s_b \in [0,\infty),  \lb{4.63} \\
&\wti u^{\, \prime}_{b,s_a,s_b}(0, a; q=0)=
\begin{cases}\dfrac{(b-a)^{s_b-s_a}\Gamma(-2s_a)\Gamma(1+2s_b)}{\Gamma((1/2)-s_a+s_b+\sigma_{s_a,s_b})\Gamma((1/2)-s_a+s_b-\sigma_{s_a,s_b})}, \\
\hfill s_a \in (0,1), \; s_b \in [0,\infty), \\
\dfrac{(b-a)^{s_b}\Gamma(1+2s_b)}{\Gamma((1/2)+s_b+\sigma_{0,s_b})\Gamma((1/2)+s_b-\sigma_{0,s_b})}[\ln(b-a)   \\
-2\gamma_{E}-\psi((1/2)+s_b+\sigma_{0,s_b})-\psi((1/2)+s_b-\sigma_{0,s_b})],\\
\hfill  s_a=0, \; s_b \in [0,\infty), 
\end{cases}  \lb{4.64} 
\end{align}
where the cases $s_b=1/2$ in \eqref{4.62} and $s_a=1/2$ in \eqref{4.64} are understood as limits. In particular, repeatedly employing $\Gamma(1+z)=z\Gamma(z)$, one obtains, 
\begin{align}
&\wti u^{\, \prime}_{a,s_a,s_b}(0, b; q=0)=\dfrac{(b-a)^{s_a-s_b}\Gamma(2+2s_a)\Gamma(2-2s_b)}{4s_b\Gamma((3/2)+s_a-s_b+\sigma_{s_a,s_b})\Gamma((3/2)+s_a-s_b-\sigma_{s_a,s_b})},   \no \\
& \hspace*{7.3cm} s_a \in [0,\infty),\;  s_b \in (0,1), \lb{4.65}  \\
&\wti u^{\, \prime}_{b,s_a,s_b}(0, a; q=0)= \dfrac{-(b-a)^{s_b-s_a}\Gamma(2-2s_a)\Gamma(2+2s_b)}{4s_a\Gamma((3/2)-s_a+s_b+\sigma_{s_a,s_b})\Gamma((3/2)-s_a+s_b-\sigma_{s_a,s_b})},   \no \\
& \hspace*{7.3cm} s_a \in (0,1),\;  s_b \in [0,\infty).  \lb{4.66} 
\end{align}
Using \eqref{4.61}--\eqref{4.66} one computes for $\alpha_K(q=0),\beta_K(q=0)$ in \eqref{4.40} and \eqref{4.41}, \begin{align}
& \cot(\alpha_K(q=0))    \no \\
& \quad =
\begin{cases}\dfrac{\Gamma(2-2s_a)\Gamma((1/2)+s_a+s_b+\sigma_{s_a,s_b})\Gamma((1/2)+s_a+s_b-\sigma_{s_a,s_b})}{\Gamma(1+2s_a)\Gamma((3/2)-s_a+s_b+\sigma_{s_a,s_b})\Gamma((3/2)-s_a+s_b-\sigma_{s_a,s_b})}  \\[3mm]
\quad \times\dfrac{(1+2s_b)}{4s_a(b-a)^{2s_a}}, \quad s_a \in (0,1),    \\[2mm]
2\gamma_{E}+\psi((1/2)+s_b+\sigma_{0,s_b})+\psi((1/2)+s_b-\sigma_{0,s_b})-\ln(b-a),  \\
\hspace*{8.6cm} s_a=0,
\end{cases}  \no \\
&\hspace*{9.05cm}  s_b \in [1,\infty),  \lb{4.67}  \\
&\cot(\beta_K(q=0))    \no \\ 
& \quad =
\begin{cases}\dfrac{\Gamma(2-2s_b)\Gamma((1/2)+s_a+s_b+\sigma_{s_a,s_b})\Gamma((1/2)+s_a+s_b-\sigma_{s_a,s_b})}{\Gamma(1+2s_b)\Gamma((3/2)+s_a-s_b+\sigma_{s_a,s_b})\Gamma((3/2)+s_a-s_b-\sigma_{s_a,s_b})}  \\[3mm]
\quad \times\dfrac{-(1+2s_a)}{4s_b(b-a)^{2s_b}}, \quad  s_b \in (0,1),   \\[2mm]
\ln(b-a)-2\gamma_{E}-\psi((1/2)+s_a+\sigma_{s_a,0})-\psi((1/2)+s_a-\sigma_{s_a,0}),  \\
\hfill s_b=0,
\end{cases}   \no   \\
&\hspace*{9.2cm}  s_a \in [1,\infty).   \lb{4.68}
\end{align}
Moreover, \eqref{4.68} recovers \eqref{4.35} whenever $s_b=1/2$, while $s_a=1/2$ in \eqref{4.67} implies the analog of \eqref{4.35} with the endpoints $a$ and $b$ interchanged.

Next, we turn to the case when both endpoints are in the limit circle case, that is, whenever $s_a,s_b\in[0,1)$. Similarly to before, \cite[no.~15.3.6, 15.3.10]{AS72} show the leading behavior of the nonprincipal solution $\hatt u_{a,s_a,s_b}(0,\dott; q=0)$ is given by
\begin{align}
& \hatt u_{a,s_a,s_b}(0,x; q=0)   \no \\
& \quad \underset{x\uparrow b}{=}
\begin{cases}
\dfrac{-(b-a)^{s_b - s_a}(b-x)^{(1/2) - s_b}\Gamma(-2s_a)\Gamma(2s_b)}{\Gamma((1/2)-s_a+s_b+\sigma_{s_a,s_b})\Gamma((1/2)-s_a+s_b-\sigma_{s_a,s_b})} \\[3mm]
\qquad \times (1+\Oh(b-x))\\[1mm]
\quad -\dfrac{(b-a)^{-s_a-s_b}(b-x)^{(1/2) + s_b}\Gamma(-2s_a)\Gamma(-2s_b)}{\Gamma((1/2)-s_a-s_b+\sigma_{s_a,s_b})\Gamma((1/2)-s_a-s_b-\sigma_{s_a,s_b})}\\[3mm]
\qquad \times (1+\Oh(b-x)),\quad s_a,s_b\in(0,1),  \\[2mm]
\dfrac{\ln(1/(b-a))(b-a)^{s_b}(b-x)^{(1/2) - s_b}\Gamma(2s_b)}{\Gamma((1/2)+s_b+\sigma_{0,s_b})\Gamma((1/2)+s_b-\sigma_{0,s_b})}(1+\Oh(b-x))  \\[3mm]
\quad +\dfrac{\ln(1/(b-a))(b-a)^{-s_b}(b-x)^{(1/2) + s_b}\Gamma(-2s_b)}{\Gamma((1/2)-s_b+\sigma_{s_b})\Gamma((1/2)-s_b-\sigma_{0,s_b})}\\[3mm]
\qquad \times (1+\Oh(b-x)),\quad s_a=0,\;  s_b\in(0,1),  \\[2mm]
\dfrac{-(b-a)^{-s_a}(b-x)^{1/2}\Gamma(-2s_a)}{\Gamma((1/2)-s_a+\sigma_{s_a,0})\Gamma((1/2)-s_a-\sigma_{s_a,0})} [\ln(b-a)-2\gamma_{E}\\[3mm]
\ \ -\psi((1/2)-s_a+\sigma_{s_a,0})-\psi((1/2)-s_a-\sigma_{s_a,0})-\ln(b-x)] \\
\qquad \times (1+\Oh(b-x)), \quad s_a\in(0,1),\;  s_b=0,  \\[2mm]
\dfrac{\ln(1/(b-a))(b-x)^{1/2}}{\Gamma((1+i)/2)\Gamma((1-i)/2)}[\ln(b-a)-2\gamma_{E}-\psi((1+i)/2)  \\[3mm]
\quad -\psi((1-i)/2)-\ln(b-x)](1+\Oh(b-x)),\quad s_a=s_b=0.
\end{cases}
\end{align}
Hence one computes the generalized boundary values $($employing $\Gamma(1+z)=z\Gamma(z)$$)$,
\begin{align}
& \wti{\hatt u}_{a,s_a,s_b}(0,b;q=0)    \no \\ 
& \quad = \begin{cases}
\dfrac{(b-a)^{s_b - s_a}\Gamma(2-2s_a)\Gamma(2+2s_b)}{4s_a\Gamma((3/2)-s_a+s_b+\sigma_{s_a,s_b})\Gamma((3/2)-s_a+s_b-\sigma_{s_a,s_b})},  \\[1mm]
\hspace*{5.45cm} s_a\in(0,1),\;  s_b\in[0,1),  \\[2mm]
\dfrac{\ln(1/(b-a))(b-a)^{s_b}\Gamma(1+2s_b)}{\Gamma((1/2)+s_b+\sigma_{0,s_b})\Gamma((1/2)+s_b-\sigma_{0,s_b})}, \quad s_a=0, \; s_b\in[0,1),
\end{cases}   \lb{4.70}  \\
& \wti{\hatt u}^{\, \prime}_{a,s_a,s_b}(0,b;q=0)    \no \\
& \quad = \begin{cases}
\dfrac{(b-a)^{-s_a-s_b}\Gamma(2-2s_a)\Gamma(2-2s_b)}{8s_a s_b \Gamma((3/2)-s_a-s_b+\sigma_{s_a,s_b})\Gamma((3/2)-s_a-s_b-\sigma_{s_a,s_b})},\\[1mm]
\hspace*{7.05cm} s_a,s_b\in(0,1),  \\[2mm]
\dfrac{\ln(1/(b-a))(b-a)^{-s_b}\Gamma(2-2s_b)}{4s_b\Gamma((3/2)-s_b+\sigma_{s_b})\Gamma((3/2)-s_b-\sigma_{0,s_b})}, \quad s_a=0,\;  s_b\in(0,1),  \\[2mm]
\dfrac{(b-a)^{-s_a}\Gamma(2-2s_a)}{4s_a\Gamma((3/2)-s_a 
+\sigma_{s_a,0})\Gamma((3/2)-s_a-\sigma_{s_a,0})} [\ln(b-a)-2\gamma_{E}\\[3mm]
\quad -\psi((1/2)-s_a+\sigma_{s_a,0})-\psi((1/2)-s_a-\sigma_{s_a,0})], \quad s_a\in(0,1),\;  s_b=0,  \\[2mm]
\dfrac{\ln(b-a)}{\Gamma((1+i)/2)\Gamma((1-i)/2)}[\ln(b-a)-2\gamma_{E}-\psi((1+i)/2)\\[3mm]
\quad -\psi((1-i)/2)], \quad s_a=s_b=0.
\end{cases}  \lb{4.71}
\end{align}

Finally, we define the Krein--von Neumann extension whenever $s_a,s_b\in[0,1)$ in \eqref{4.43} by choosing
\begin{align}
R_K(q=0) = \begin{pmatrix} \wti{\hatt u}_{a,s_a,s_b}(0,b; q=0) & \wti u_{a,s_a,s_b}(0,b;q=0)\\
\wti{\hatt u}_{a,s_a,s_b}^{\, \prime}(0,b;q=0) & \wti u_{a,s_a,s_b}^{\, \prime}(0,b;q=0)
\end{pmatrix},\quad s_a,s_b\in[0,1),  \lb{4.72}
\end{align}
in \eqref{4.44} with the generalized boundary values given by \eqref{4.61}, \eqref{4.62}, \eqref{4.65}, \eqref{4.70}, and \eqref{4.71}.
Moreover, \eqref{4.72} recovers \eqref{4.37} whenever $s_b=1/2$ and gives the analog of \eqref{4.37} with the endpoint behavior interchanged whenever $s_a=1/2$. 
\end{example}

\section{On Hardy-Type Inequalities} \lb{sB} 

Recalling the differential expressions $\alpha_{s_a}$, $\alpha^+_{s_a}$ in \eqref{3.1}, \eqref{3.2}, we follow the Hardy inequality considerations in \cite{GP79}, \cite{Ka72}, \cite{KW72} to obtain the following basic facts.

\begin{lemma} \lb{lB.1} 
Suppose $f \in AC_{loc}((a,b))$, $\alpha_{s_a} f \in L^2((a,b); dx)$ for some $s_a \in\bbR$, and 
$a < r_0 < r_1 \leq b < R < \infty$. Then, 
\begin{align}
& \int_{r_0}^{r_1} dx \, |(\alpha_{s_a} f)(x)|^2 \geq s_a^2 \int_{r_0}^{r_1} dx \, \f{|f(x)|^2}{(x-a)^2} + 
\f{1}{4} \int_{r_0}^{r_1} dx \, \f{|f(x)|^2}{(x-a)^2 [\ln(R/(x-a))]^2}   \no \\
& \hspace*{3.25cm} - s_a \f{|f(x)|^2}{(x-a)}\bigg|_{x = r_0}^{r_1} - \f{|f(x)|^2}{2 (x-a) [\ln(R/(x-a))]}\bigg|_{x = r_0}^{r_1},   \lb{B.1} \\
& \int_{r_0}^{r_1} dx \, (x-a) \ln(R/(x-a)) \bigg|\bigg[\f{f(x)}{(x-a)^{1/2} [\ln(R/(x-a))]^{1/2}}\bigg]'\bigg|^2     \no \\
& \quad = \int_{r_0}^{r_1} dx \, \bigg[|f'(x)|^2 - \f{|f(x)|^2}{4 (x-a)^2} 
- \f{|f(x)|^2}{4 (x-a)^2 [\ln(R/(x-a))]^2}\bigg]     \lb{B.3} \\
& \qquad - \f{|f(x)|^2}{2(x-a)}\bigg|_{x = r_0}^{r_1} + \f{|f(x)|^2}{2(x-a) \ln(R/(x-a))}\bigg|_{x = r_0}^{r_1} \geq 0,   \no \\
\begin{split} 
& \int_{r_0}^{r_1} dx \, |(\alpha_{s_a} f)(x)|^2 =  \int_{r_0}^{r_1} dx \, \bigg[|f'(x)|^2 
+ \big[s_a^2 - (1/4)\big] \f{|f(x)|^2}{(x-a)^2}\bigg]     \lb{B.4} \\
& \hspace*{3.4cm} - [s_a + (1/2)] \f{|f(x)|^2}{(x-a)}\bigg|_{x = r_0}^{r_1} \geq 0.    
\end{split} 
\end{align}
If $s_a = 0$, then 
\begin{equation}
\int_a^{r_1} dx \, \f{|f(x)|^2}{(x-a)^2 [\ln(R/(x-a))]^2} < \infty, \quad  
\lim_{x \downarrow a} \f{|f(x)|}{[(x-a) \ln(R/(x-a))]^{1/2}} = 0.    \lb{B.5} 
\end{equation}
If $s_a \in (0,\infty)$, then  
\begin{equation}
\int_a^{r_1} dx \, |f'(x)|^2 < \infty, \quad \int_a^{r_1} dx \, \f{|f(x)|^2}{(x-a)^2} < \infty, \quad 
\lim_{x \downarrow a} \f{|f(x)|}{(x-a)^{1/2}} = 0,     \lb{B.6} 
\end{equation}
in particular, for $0 < \varepsilon$ sufficiently small, 
\begin{equation}
f \wti \chi_{[a,a+\varepsilon]} \in H^1_0((a,r_1)).
\end{equation}
\end{lemma}
\begin{proof}
Relations \eqref{B.3} and \eqref{B.4} are straightforward (yet somewhat tedious) identities; together 
they yield \eqref{B.1}. The first relation in \eqref{B.5} is an instant consequence of \eqref{B.1}, so is the fact that 
$\lim_{x \downarrow a} |f(x)|^2 / [(x-a) \ln(R/(x-a))]$ exists. Moreover, since $[(x-a) \ln(R/(x-a))]^{-1}$ is not 
integrable at $x = a$, the first relation in \eqref{B.5} yields 
$\liminf_{x \downarrow a} |f(x)|^2 / [(x-a) \ln(R/(x-a))] = 0$, implying the second relation in \eqref{B.5}.  

Finally, if $s_a \in (0,\infty)$, then $\alpha_{s_a} f \in  L^2((a,b); dx)$ and inequality \eqref{B.1} implies the second 
relation in \eqref{B.6}; employing $\alpha_{s_a} f \in  L^2((a,b); dx)$ once more yields the first relation in \eqref{B.6}. 
By inequality \eqref{B.1}, $\lim_{x \downarrow a} |f(x)|^2 /(x-a)$ exists, but then the second relation in \eqref{B.6} yields 
$\liminf_{x \downarrow a} |f(x)|^2 /(x-a) = 0$ and hence also $\lim_{x \downarrow a} |f(x)|^2/(x-a) = 0$, the third 
relation in \eqref{B.6}. 
\end{proof}

We continue with the following elementary fact.

\begin{lemma} \lb{lB.2}
Suppose $r \in (0,\infty) \cup \{\infty\}$, $f \in AC_{loc}((a,r))$ and $f' \in L^2((a,r); dx)$. Then, for all 
$c, x \in (a,r)$, 
\begin{align}
\begin{split} 
|f(x) - f(c)| &= \bigg|\int_c^x dt \, f'(t)\bigg| \leq (x - c)^{1/2} \bigg(\int_c^x dt \, |f'(t)|^2\bigg)^{1/2}    \\
& \leq (x - c)^{1/2} \|f'\|_{L^2((a,x);dt)} \leq (x - c)^{1/2} \|f'\|_{L^2((a,r);dt)}.     \lb{B.9}
\end{split} 
\end{align} 
In particular, $f(a) = \lim_{c \downarrow a} f(c)$ exists and hence $f \in AC([a,d])$ for all $d \in (a,r)$ and 
\begin{equation}
|f(x) - f(a)| \underset{x \downarrow a}{=} \oh\big((x-a)^{1/2}\big).    \lb{B.10}
\end{equation}
Finally, if $r \in (a,\infty)$, then also $f(r) = \lim_{x \uparrow r} f(x)$ exists and hence $f \in AC([a,r])$.  
\end{lemma} 

A particular consequence of \eqref{B.9} is the fact that for $a, b \in \bbR$, $a < b$, the conditions $ f(a)=0=f(b)$ in \eqref{B.10a} below, 
\begin{align}
\begin{split} 
& H^1_0((a,b)) = \big\{f\in L^2((a,b);dx) \, \big| \, f \in AC([a,b]); \, f(a)=0=f(b);   \\
& \hspace*{7.2cm}  f' \in L^2((a,b);dx)\big\},       \lb{B.10a}
\end{split} 
\end{align}
can actually be replaced by
\begin{equation}
\lim_{x \downarrow a} \f{f(x)}{(x-a)^{1/2}} = 0 = \lim_{x \uparrow b} \f{f(x)}{(b-x)^{1/2}}.   \lb{B.10b}
\end{equation}

We also recall the following result (and its proof), borrowed, for instance, from \cite[Lemma~5.3.1]{Da95}: 

\begin{lemma} \lb{lB.3}
Suppose $r \in (a,\infty) \cup \{\infty\}$,  $f \in AC_{loc}((a,r))$, $f' \in L^2((a,r);dx)$, $f(a) = 0$. Then 
\begin{equation}
 \int_a^r dx \, |f'(x)|^2 \geq \f{1}{4} \int_a^r dx \, \f{|f(x)|^2}{(x-a)^2}.       \lb{B.11} 
\end{equation}
In particular, if $r \in (a,\infty)$, no boundary conditions are needed in $f$ at the right end point $r$ in order to derive this version of Hardy's inequality; moreover, for $0 < \varepsilon$ sufficiently small,
\begin{equation}
f \wti \chi_{[a,a+\varepsilon]} \in H^1_0((a,r)). 
\end{equation}
If $r = \infty$, then $f \in H^1_0((a,\infty))$.
\end{lemma}
\begin{proof}
Without loss of generality we may assume that $f$ is real-valued. Then, by the argument in 
\cite[Lemma~5.3.1]{Da95} (see also \cite[Sect.~1.1]{OK90} for generalizations), 
\begin{align}
& \int_a^r dx \, |f'(x)|^2 = \int_a^r dx \, \big|(x-a)^{1/2} \big[(x-a)^{-1/2} f(x)\big]' + [2(x-a)]^{-1} f(x)\big|^2 \no \\
& \quad = \int_0^r dx \, \Big\{4^{-1} (x-a)^{-2} f(x)^2 + (x-a)^{-1/2} f(x) \big[(x-a)^{-1/2} f(x)\big]'     \no \\
& \qquad + (x-a) \big[\big((x-a)^{-1/2} f(x)\big)'\big]^2\Big\}    \no \\
& \quad \geq \int_a^r dx \, \Big\{4^{-1} (x-a)^{-2} f(x)^2 + (x-a)^{-1/2} f(x) \big[(x-a)^{-1/2} f(x)\big]'\Big\}    \no \\ 
& \quad =  \int_a^r dx \, \f{|f(x)|^2}{4 (x-a)^2} + 2^{-1} \big[(x-a)^{-1/2} f(x)\big]^2 \big|_{x=a}^r   \no \\
& \quad \geq  \int_a^r dx \, \f{|f(x)|^2}{4 (x-a)^2} - 2^{-1} \lim_{x \downarrow a}\big[(x-a)^{-1/2} f(x)\big]^2 \no \\ 
& \quad =  \int_a^r dx \, \f{|f(x)|^2}{4 (x-a)^2}, \quad r \in (a,\infty) \cup \{\infty\},       \lb{B.13} 
\end{align} 
employing the estimate \eqref{B.10} with $f(a)=0$. 

In fact, once more by estimate \eqref{B.10}, one also notes that the hypotheses on $f$ actually imply that 
$f \in AC([a,c])$, $c \in (a,r)$, and hence that $f$ behaves like an $H^1_0$-function in a right neighborhood of $x = a$, equivalently, $f \wti \chi_{[a,a+\varepsilon]} \in H^1_0((a,r))$. If $r = \infty$, $f \in AC([a,c])$ for all $c \in (a,\infty)$ and 
$f' \in L^2((a,\infty); dx)$ is well-known to be equivalent to $f \in H^1_0((a,\infty))$. 
\end{proof}

It is well-known that the constant $1/4$ in \eqref{B.11} is optimal. 

\begin{corollary} \lb{cB.4}
Suppose that $r \in (a,\infty) \cup \{\infty\}$, $f \in AC_{loc}((a,r))$, and $f' \in L^2((a,r); dx)$. Then 
\begin{equation}
\int_a^r dx \, |f(x)|^2/(x-a)^2 < \infty \, \text { if and only if } \, f(a) = 0,   \\
\end{equation}
equivalently, 
\begin{equation} 
\int_a^r dx \, |f(x)|^2/(x-a)^2 < \infty \, \text { if and only if } \, f \wti \chi_{[a,a+\varepsilon]} \in H^1_0((a,r)). 
\end{equation}
In particular, if $f \in H^1((a,r))$ and $f(a) = 0$, then actually,  
\begin{equation}
\lim_{x \downarrow a} \f{|f(x)|}{(x-a)^{1/2}} = 0.     \lb{B.16} 
\end{equation}
\end{corollary} 
\begin{proof}
If $f \in AC_{loc}((a,r))$, $f' \in L^2((a,r); dx)$, and $\int_a^r dx \, |f(x)|^2/(x-a)^2 < \infty$, then inequality \eqref{B.4} for 
$s_a < - 1/2$, that is, 
\begin{equation}
\int_{r_0}^{r_1} dx \, \bigg[|f'(x)|^2 + \big[s_a^2 - (1/4)\big] \f{|f(x)|^2}{(x-a)^2}\bigg]  \geq 
[s_a + (1/2)] \f{|f(x)|^2}{(x-a)}\bigg|_{x = r_0}^{r_1}, \quad s_a < - 1/2,    \lb{B.17} 
\end{equation} 
yields the existence of $\lim_{x \downarrow a} |f(x)|^2/(x-a)$. Since $\int_a^r dx \, |f(x)|^2/(x-a)^2 < \infty$ 
implies that  $\liminf_{x \downarrow a} |f(x)|^2/(x-a) = 0$, one concludes that $\lim_{x \downarrow a} |f(x)|^2/(x-a) = 0$ 
and hence $f$ behaves locally like an $H^1_0$-function in a right neighborhood of $x = a$. Conversely, if 
$f \in AC_{loc}((a,r))$, $f' \in L^2((a,r); dx)$, and $f(a) = 0$, then $\int_a^r dx \, |f(x)|^2/(x-a)^2 < \infty$ by Hardy's inequality as discussed in Lemma \ref{lB.3}. Relation \eqref{B.16} is clear from \eqref{B.10} with $f(a) = 0$. 
\end{proof}

\begin{remark} \lb{rB.5} 
$(i)$ If $f \in AC_{loc}((a,r))$ and $f' \in L^p((a,r); dx)$ for some $p \in [1,\infty)$, $r \in (0,\infty)$, the H\"older estimate analogous to \eqref{B.9},
\begin{align}
\begin{split} 
|f(d) - f(c)| = \bigg|\int_c^d dt \, f'(t) \bigg| \leq |d-c|^{1/p'} \bigg(\int_c^d dt \, |f'(t)|^p\bigg)^{1/p}, \\ 
(c,d) \subset (0,r), \; \f{1}{p} + \f{1}{p'} = 1, 
\end{split} 
\end{align} 
implies the existence of $\lim_{c \downarrow a} f(c) = f(a)$ and $\lim_{d \uparrow r} f(d) = f(r)$ and hence yields   
$f \in AC([a,r])$. \\[1mm] 
$(ii)$ The fact that $f \in H^1((a,r))$ and $\int_a^r dx \, |f(x)|^2/(x-a)^2 < \infty$ implies $f \wti \chi_{[a,a+\varepsilon]} \in H^1_0((a,r))$ if $r \in (a,\infty)$, and $f \in H^1_0((a,\infty))$ if $r=\infty$, is a special case of a multi-dimensional result recorded, for instance, in \cite[Theorem~5.3.4]{EE18}. 
${}$ \hfill $\diamond$ 
\end{remark}

\medskip

\noindent {\bf Acknowledgments.}
We are indebted to Larry Allen and Jan Derezinski for very helpful discussions. 


\end{document}